\documentclass[final]{siamltex}

\newcommand{\Z}{\mathbb{Z}}
\newcommand{\R}{\mathbb{R}}
\newcommand{\N}{\mathbb{N}}

\newcommand{\X}{\mathcal{X}}

\newcommand{\LM}{\mathcal{L}}

\newcommand{\dt}{\delta}
\newcommand{\Dt}{\Delta}

\newcommand{\D}{\mathcal{D}}

\newcommand{\e}{\epsilon}

\newcommand{\E}{\mathcal{E}}

\newcommand{\V}{\mathcal{V}}
\newcommand{\Gr}{\mathcal{G}}
\newcommand{\W}{\mathcal{W}}
\newcommand{\s}{\sigma}
\newcommand{\Sig}{\Sigma}

\newcommand{\row}{\operatorname{Row}}
\newcommand{\col}{\operatorname{Col}}
\newcommand{\outdeg}{\operatorname{outdeg}}
\newcommand{\indeg}{\operatorname{indeg}}
\newcommand{\temp}{\operatorname{temp}}

\newtheorem{example}[theorem]{Example}
\newtheorem{remark}{Remark}[section]
\usepackage{times,amsmath,amssymb,graphicx,tikz,cite,enumerate,algorithm,algorithmic}
\usetikzlibrary{automata,arrows,positioning}
\usepackage[colorlinks,linkcolor=blue,hyperindex]{hyperref}

\def\e{{\epsilon}}

\title{A weighted pair graph representation for reconstructibility of
Boolean control networks\thanks{
This work was supported by
National Natural Science Foundation of China (No. 61603109), Fundamental Research Funds for the Central Universities
(No. HEUCF160404), Natural Science Foundation of Heilongjiang Province of China (No. LC2016023),
National Natural Science Foundation of China (No. 61573288), Program for New Century Excellent Talents
in University of Ministry of Education of China, 
and Singapore Ministry of Education Tier 1 Academic Research Grant RG84/13 (2013-T1-002-177).
}}

\author{Kuize Zhang\thanks{College of Automation, Harbin Engineering University, Harbin, 150001, PR China
 and
Department of Electrical and Computer Engineering, Technical University of Munich, 80333 Munich, Germany
	({\tt zkz0017@163.com}).}
        \and Lijun Zhang\thanks{School of Marine Science and Technology, Northwestern Polytechnical University, Xi'an, 710072, PR China
        and College of Automation, Harbin Engineering University, Harbin, 150001, PR China
         ({\tt zhanglj7385@nwpu.edu.cn}).}
         \and Rong Su\thanks{
         School of Electrical and Electronic Engineering, Nanyang Technological University, Singapore 639798, Singapore
         	({\tt rsu@ntu.edu.sg}).}
		}

\begin{document}

\maketitle

\begin{abstract}
	A new concept of weighted pair graphs (WPGs) is proposed to represent a new reconstructibility
	definition for Boolean control networks (BCNs), which is a generalization of the reconstructibility
	definition given in \cite[Def. 4]{Fornasini2013ObservabilityReconstructibilityofBCN}.
	Based on the WPG representation, an effective algorithm for determining the new reconstructibility
	notion for BCNs is designed with the help of the theories of finite automata and formal languages. We
	prove that a BCN is not reconstructible iff its WPG has a complete subgraph. Besides, we prove that a BCN is
	reconstructible in the sense of \cite[Def. 4]{Fornasini2013ObservabilityReconstructibilityofBCN}
	iff its WPG has no cycles, which is simpler to be checked
	than the condition in \cite[Thm. 4]{Fornasini2013ObservabilityReconstructibilityofBCN}.
\end{abstract}



\pagestyle{myheadings}
\thispagestyle{plain}
\markboth{KUIZE ZHANG AND LIJUN ZHANG AND RONG SU}{A WEIGHTED PAIR GRAPH REPRESENTATION FOR RECONSTRUCTIBILITY}
\section{Introduction}

{\it Reconstructibility} is a basic control-theoretic property. One way it can
be formulated as the property that there is an input sequence such that the current state can be
uniquely determined from the input sequence and the corresponding output sequence, regard-
less of the initial state, which may be considered unknown. Accordingly the key problems
are how to determine whether such input sequences exist and how to find them. Note that for
deterministic systems, once the current state has been determined, all subsequent states can
also be determined by using the input sequences.
As an application, reconstructibility can be used in fault detection for a
mechanical device if the device is reconstructible. If one regards a device as a control system,
and regards the states (resp. events) of the device as the states (resp. outputs) of the system,
then fault detection can be implemented by using event sequences to determine the current
state of the device.

A {\it Boolean (control) network} (BN/BCN) (cf.
\cite{Kauffman1969RandomBN,Ideker(2001),Akutsu2000InferringQualitativeRelations,Albert2000DynamicsofComplexSystems,Kitano2002SystemsBiology,Laschov2013Minimum-timecontrolBN,Zhang2015Invertibility_BCN}),
a discrete-time finite-state dynamical system,
is a simple and effective model to describe genetic regulatory networks (GRNs)
which reflects the behavior and relationships of cells, protein, DNA and RNA in a biological
system. It is pointed out in
\cite{Akutsu(2007)} that
that ``One of the major goals of systems biology is to develop a control theory for complex biological systems''.
Hence studying the control-theoretic
problems of BNs/BCNs is of both theoretical and practical importance. Similarly to the fault
detection for a mechanical device, reconstructibility may also be used in biology to detect
diseases in a living body. In
\cite{Sridharan2012BNFaultDiagnosisOSR},
fault diagnosis in oxidative stress response is investigated
based on a BCN model, and a fault is described as a deviation of the function of the BCN
model. In order to diagnose a fault, two steps should be performed successively: \romannumeral1) use an
input sequence (called homing sequence) to drive the model to a known state; (If the model
is normal and reconstructible, then after feeding the homing sequence into the model, the
current state of the model would be known. It is pointed out in
\cite{Sridharan2012BNFaultDiagnosisOSR}
that ``Knowledge of the
initial status of the internal states is important as all future computations are based on these
values. The Homing sequence is an initial input sequence that brings the network to a known
internal state. So, once the Homing Sequence is given to $N$ (the normal model) and $N_f$ (the
faulty model), $N$ will come to a known internal state.'') \romannumeral2) feed an input sequence (called test
sequence) into the model, compare the output sequences of the normal and faulty networks
to pinpoint the fault. (``Once the Homing sequence has done its job, the Test sequence($T$) is
fed into $N$ and $N_f$, and by comparing the output states of the normal and faulty networks, we
can pinpoint the location of the fault in the network, \dots''. One can use the methods adopted in
\cite{Sridharan2012BNFaultDiagnosisOSR} to pinpoint the fault. Besides, methods for testing fault occurrence
are given in \cite{Fornasini2015FaultDetectionBCN}.)
Based on the above statement, reconstructibility is the first step for diagnosing faults, and the
current state plays an important role. Hereinafter, ususally we will call the input sequence described
in the first and this paragraph that can be used to determine the current state {\it homing}
input sequence for short.
\cite{Fornasini2013ObservabilityReconstructibilityofBCN}
introduces a special reconstructibility notion (see
\cite[Definition 4]{Fornasini2013ObservabilityReconstructibilityofBCN}) 
for BCNs which means each sufficiently long input sequence is a homing input
sequence. If a BCN satisfies this definition, when diagnosing faults, one only needs to pay
attention to the second step; otherwise this definition will tell the user that the BCN is not
reconstructible, indicating that the fault cannot be diagnosed. However, even if a BCN does
not satisfy this definition, there still may exist a homing input sequence, and the user can use
it to reconstruct the current state. Hence it is necessary to investigate the reconstructibility
described in the first paragraph involving whether a homing input sequence exists. The reconstructibility
described in the first paragraph is more general, as it applies to more systems.
Besides, as the one in 
\cite[Definition 4]{Fornasini2013ObservabilityReconstructibilityofBCN},
this reconstructibility definition is also independent of the initial state, so any time can be seen as
the initial time.


The main target of this paper is to design an effective algorithm for determining the reconstructibility 
of BCNs described in the first paragraph. In the sequel, unless otherwise stated, ``reconstructibility''
is always in this sense. The original idea of designing this algorithm comes from our previous paper
\cite{Zhang2014ObservabilityofBCN}.
In \cite{Zhang2014ObservabilityofBCN},
we find the connections between the observability of BCNs and the theories of finite automata and formal languages,
and show how to determine all known four different types of observability of BCNs in the literature
\cite{Cheng2009bn_ControlObserva,Zhao2010InputStateIncidenceMatrix,Cheng2011IdentificationBCN,Fornasini2013ObservabilityReconstructibilityofBCN,Li2013ObservabilityConditionsofBCN,Li2015ControlObservaBCN,Cheng2015NoteonObservabilityBCN}.
In particular, the type of observability first studied in the seminal paper 
\cite{Cheng2009bn_ControlObserva} is determined in \cite{Zhang2014ObservabilityofBCN},
while in \cite{Cheng2009bn_ControlObserva} there is only a sufficient but not necessary
condition.
In our companion paper \cite{Zhang2016ObservabilitySBCN_NAHS},
this idea is also used to determine the observability of switched BCNs,
and further results on how to reduce computational complexity is discussed.
The theories of finite automata and formal languages are among the mathematical foundations of
theoretical computer science. Finite automaton theory involves mainly the study of computational problems
that can be solved by using them. In computational complexity theory, decision problems are typically
defined as formal languages, and complexity classes are defined as the sets of the formal languages
that can be parsed by machines with limited computational power. For the details, we refer the reader to
\cite{Kari2013LectureNoteonAFL},
\cite{Linz2006FormalLanguagesandAutomata}. In the control-theoretic field,
finite automata have been used to describe
discrete event systems (DESs) (cf. \cite{Ramadge1987SupervisoryControlofDES,Lin1988Observability_DES,Wonham1976InternalModelPriniple,Shu2007Detectability_DES,Wonham2014SupervisoryControlDES,Su2011StringExecutionTime,Cassandras2008DES_monograph,Lin2014LinControlofNetworkedDES}, etc.),
where DESs are event-driven systems, and have no normal time sequences,
which are essentially different from the standard input-state-output control system models.


In order to characterize reconstructibility, we first define a new concept of {\it weighted pair graphs} for 
BCNs\footnote{This weighted pair graph is different
from the one defined in \cite{Zhang2014ObservabilityofBCN} used to connect observability of BCNs
and finite automata.},
second we use the graph to transform a BCN into a deterministic finite automaton,
and lastly we test its reconstructibility by verifying the completeness of the automaton.
Using these results, once we know that a BCN is reconstructible, each homing input
sequence can be found. After that we design an algorithm to use a homing input sequence to determine the current state.
Furthermore, after proving more in-depth results on the weighted pair graph by using finite automata,
we directly use the graph to design a remarkably more effective algorithm to determine reconstructibility (see Section
\ref{sec:EffectiveAlgorithmForReconstructibility}).
On the other hand, in Section \ref{sec:Fornasini'sReconstructibility},
as a comparison, we prove that the weighted pair graph has no cycles iff the BCN is reconstructible in the sense of 
\cite[Definition 4]{Fornasini2013ObservabilityReconstructibilityofBCN}, which is simpler to be checked
than the condition in \cite[Theorem 4]{Fornasini2013ObservabilityReconstructibilityofBCN}.

The remainder of this paper is organized as follows. In Section \ref{sec:Preliminaries},
necessary preliminaries about {\it graph theory}, finite automata,
formal languages, the {\it semi-tensor product (STP) of matrices}, and BCNs with
their algebraic forms are introduced.
By using STP, a BCN can be transformed into its algebraic form.
Such an intuitive algebraic form will help to represent
weighted pair graphs and finite automata constructed in the sequel, and also
help to construct examples.
So this paper is in the framework of STP.
A comprehensive introduction to the STP of matrices can be found in
\cite{Cheng2010bn_dynamics,Cheng(2011book)}.
In Section \ref{sec:DeterminingReconstructibility}, how to use finite automata
to determine the reconstructibility of BCNs and how to use homing input sequences
to determine the current state are illusatrated.
Section \ref{sec:EffectiveAlgorithmForReconstructibility} contains 
the main results: we
directly use the weighted pair graph to design
a remarkably more effective algorithm to determine the reconstructibility of BCNs
and analyze its computational complexity. 
In Section \ref{sec:Fornasini'sReconstructibility}, an intuitive algorithm for
determining the reconstructibility shown in
\cite[Definition 4]{Fornasini2013ObservabilityReconstructibilityofBCN} is designed.
The last section is a short conclusion.

\section{Preliminaries}\label{sec:Preliminaries}

Necessary notations:

\begin{itemize}
  \item $\emptyset$: the empty set
  \item $\R_{m\times n}$: the set of $m\times n$ real matrices
  \item $\Z_+$: the set of positive integers (excluding $0$)
  \item $\N$: the set of natural numbers (including $0$)
  \item $\D$: the set $\{0,1\}$
  \item $\delta_n^i$: the $i$-th column of the $n\times n$ identity matrix $I_n$
  \item $\Delta_n$: the set $\{\delta_n^1,\dots,\delta_n^n\}$
	  ($\Dt:=\Dt_2$)
  \item $\col_i(A)$ (resp. $\row_i(A)$): the $i$-th column (resp. row) of matrix $A$
  \item $i\mod j$: the remainder of integer $i$ when divided by integer $j$
  \item $\delta_n[i_1,\dots,i_s]$: {\it logical matrix} (see
	  \cite{Cheng2010bn_dynamics,Cheng(2011book)})
	  $[\delta_n^{i_1},\dots,\delta_n^{i_s}]$ ($i_1,\dots,i_s\in\{1,2,\dots, n\}$)
  \item $\LM_{n\times s}$: the set of $n\times s$ logical matrices, i.e.,
	  $\{\delta_n[i_1,\dots,i_s]|i_1,\dots,i_s\in\{1,2,\dots,n\}\}$
  \item $[M,N]$: the set of consecutive integers $M,M+1,\dots,N$
  \item $|A|$: the cardinality of set $A$
  \item $2^A$: the power set of set $A$
  \item $C_n^i$: binomial coefficient
\end{itemize}

\subsection{Graph theory}

In this subsection we introduce some basic concepts of graph theory.

A {\it directed graph} is a $2$-tuple $(\V,\E)$, where a finite set $\V$ denotes its {\it vertex}
set, and $\E\subset \V\times \V$ denotes its {\it edge} set. Given two vertices
$v_1,v_2\in\V$, if $(v_1,v_2)\in\E$, then we say ``there is an edge from $v_1$ to $v_2$'',
and denote $(v_1,v_2)$ also by $v_1\to v_2$.
$v_1$ is called a parent of $v_2$, and similarly $v_2$ is called a child
of $v_1$.
Given $v_0,v_1,\dots,v_{p}\in\V$, if for all $i\in[0,p-1]$,
$(v_i,v_{i+1})\in\E$, then $v_0\rightarrow \dots\rightarrow v_p$ is called
a {\it path}.
Particularly if $v_0=v_p$, path $v_0\rightarrow \dots\rightarrow v_p$
is called a {\it cycle}. A cycle $v_0\rightarrow\dots\rightarrow v_p$ is called {\it simple},
if $v_0,\dots,v_{p-1}$ are pairwise different.
An edge from a vertex to itself is called a {\it self-loop}.
Given vertices $v_0,\dots,v_p\in\V$, and denote $\{v_0,\dots,v_p\}$ by $\V_p$.
The {\it subgraph} of graph $(\V,\E)$ generated by $\V_p$ is defined as graph $(\V_p,\E_p)$,
where $\E_p=(\V_p\times\V_p)\cap \E$.
A directed graph $(\V,\E)$ is called {\it strongly connected}, if for all vertices $u,v\in\V$,
there is a path from $u$ to $v$.

Given a set $\Sig$, a weighted directed graph $(\V,\E,\W,2^{\Sig})$ is a directed graph $(\V,\E)$ such that
each edge $e\in\E$ is labeled by a {\it weight} $w\subset \Sig$, represented by a function
$\W:\E\to 2^{\Sig}$.  Given vertex $v\in\V$, $|\cup_{u\in\V,(v,u)\in\E}\W( (v,u))|$
is called the {\it outdegree} of vertex $v$, and is denoted by $\outdeg(v)$; similarly
$|\cup_{u\in\V,(u,v)\in\E}\W( (u,v))|$
is called the {\it indegree} of vertex $v$, and is denoted by $\indeg(v)$.
A subgraph is called complete, if each of its vertices has outdegree $|\Sig|$.

\subsection{Finite automata and formal languages}

We use $\Sig$, a nonempty finite set to denote the {\it alphabet}\footnote{
A nonempty finite set is an alphabet iff for each finite sequence $u$ of its elemetns,
any other finite sequence of its elements is not the same as $u$.
For example, $\{0,01\}$ is an alphabet, but $\{0,00\}$
is not, as  $000=0$ $00=00$ $0$.}.
Elements of $\Sig$ are called {\it letters}. A {\it word} is a finite sequence of letters.
The empty word is denoted by $\e$.
$|\cdot|$ denotes the length of word $\cdot$. For example, $|abc|=3$ over
alphabet $\{a,b,c\}$, $|\e|=0$.
$\Sig^p$ denotes the set of words of length $p$ over alphabet $\Sig$.
In particularly, $\Sig^0:=\{\e\}$. Hence
$\cup_{i=0}^{\infty}\Sig^i=\Sig^*$ denotes the set of all words over alphabet
$\Sig$. For example,
$$\{0,1\}^*=\{\e,0,1,00,01,10,11,000,\dots\}.$$
The set of infinite sequences of letters over alphabet $\Sig$ is denoted by $\Sig^{\N}$.
That is, $\Sig^{\N}=\{a_0a_1\dots|a_i\in\Sig,i=0,1,\dots\}$.
Each $u\in\Sig^{\N}$  satisfies $|u|=\infty$. 
Given $u\in(\Sig^*\cup\Sig^{\N})\setminus\{\e\}$ and integers $i,j$ satisfying
$0\le i\le j<|u|$, $u(i)$ or $u[i]$ denotes the $i$-th letter of $u$,
$u[i,j]$ denotes word $u(i)u(i+1)\dots u(j)$.
A {\it(formal) language} is defined as a subset of $\Sig^*$.

Next we introduce the concepts of deterministic finite automata (DFAs) and regular languages.
A DFA is a 5-tuple $A=(S,\Sig,\s,s_0,F)$:
\begin{itemize}
	\item The finite state set $S$. At all times the internal state is
		some $s\in S$.
	\item The input alphabet $\Sig$. The automaton reads only words
		over the alphabet.
	\item The transition partial function
		describes how the automaton changes its
		internal state. It is a partial function
		$$\s:S\times\Sig\to S$$
		that maps a (state, input letter)-pairs to a state, that is,
		$\s$ is a function defined on a subset of $S\times\Sig$.
		If the automaton is in state $s$, the current input
		letter is $a$,
		then the automaton changes its internal state to
		$\s(s,a)$ and moves to the next input letter, if $\s$ is
		well defined at $(s,a)$; and stops, otherwise.
	\item The initial state $s_0\in S$ is the internal state of the automaton
		before any letter has been read.
	\item The set $F\subset S$ of final states specifies which states
		are accepted and which are rejected. If
		the internal state of the automaton, after reading the whole
		input, is some state of $F$ then the word is accepted,
		otherwise rejected.
\end{itemize}

We call a
DFA complete if $\s$ is a function from $S\times\Sig$ to $S$.

In order to represent regular languages, we introduce an extended
transition function $\s^*:S\times \Sig^*\to S$. $\s^*$ is recursively
defined as
\begin{itemize}
	\item $\s^*(s,\e)=s$ for all $s\in S$.
	\item $\s^*(s,wa)=\s(\s^*(s,w),a)$ for all $s\in S$, $w\in\Sig^*$
		and $a\in\Sig$, if $\s^*$ is well defined at $(s,w)$
		and $\s$ is well defined at $(\s^*(s,w),a)$.
\end{itemize}

Particularly, for all $s\in S$ and $a\in\Sig$, $\s^*(s,a)=
\s(\s^*(s,\e),a)=\s(s,a)$, if $\s$ is well defined at $(s,a)$.
Hence we will use $\s$ to denote $\s^*$ briefly,
as no confusion will occur. 

Given a DFA $A=(S,\Sig,\s,s_0,F)$, a word $w\in\Sig^*$ is called {\it accepted} by
this DFA, if $\s(s_0,w)\in F$. A language $L\subset\Sig^*$ is called {\it recognized}
by this DFA, if $L=\{w\in\Sig^*|\s(s_0,w)\in F\}$, and is denoted by $L(A)$.
That is, a DFA $A$ is a finite description of a regular language $L(A)$.
It is easy to see that a DFA accepts the empty word $\e$ iff its initial state
is final.
It is worth noting that not all languages are regular. For example, language
$\{a^ib^i|i\ge0\}$ (see \cite[Example 24]{Kari2013LectureNoteonAFL})
that contains all words that begin with an arbitrary number of $a$'s, followed
by equally many $b$'s is not regular.

In order to represent a DFA and transform a BCN into a DFA related
to its reconstructibility, we introduce the {\it transition graph} of a DFA
$A=(S,\Sig,\s,s_0,F)$.
A weighted directed graph $G_A=(V,E,W)$ is called the transition graph of the DFA $A$,
if the vertex set $V=S$, the edge set $E\subset V\times V$ and the weight
function $W:E\to 2^\Sig$ satisfy the following conditions:
for all $(s_i,s_j)\in V\times V$,
$(s_i,s_j)\in E$ iff there is a letter $a\in\Sig$ such that $\s(s_i,a)=s_j$;
if  $(s_i,s_j)\in E$, then $W( (s_i,s_j))$ equals the set of letters
$a\in\Sig$ such that $\s(s_i,a)=s_j$, that is, $\{a\in\Sig|\s(s_i,a)=s_j\}$.
In the transition graph of a DFA, an input arrow is added to the vertex
denoting the initial state, double circles or rectangles are used to denote the
final states, the curly bracket ``$\{\}$''
in the weights of edges are not drawn. See the following examples.

\begin{example}\label{exam1_observability}
	The graph in Fig. \ref{fig1:observability} represents DFA
	$A=(\{s_0,s_1,s_2\},\{0,1\},\s,s_0,
	\{s_0,s_1\})$, where
	\begin{equation*}
		\begin{split}
			&\s(s_0,0)=s_0,\quad \s(s_1,0)=s_0,\quad \s(s_2,0)=s_2,\\
			&\s(s_0,1)=s_1,\quad \s(s_1,1)=s_2,\quad \s(s_2,1)=s_1.
		\end{split}
	\end{equation*}

	It is easy to see that $\e\in L(A)$, $010111\in L(A)$ and $010110\notin L(A)$.

	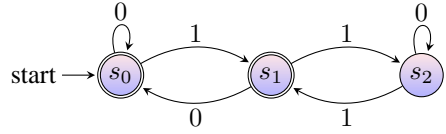
\begin{figure}
        \centering
\begin{tikzpicture}[>=stealth',shorten >=1pt,auto,node distance=2 cm, scale = 1, transform shape,
	->,>=stealth,inner sep=2pt,state/.style={shape=circle,draw,top color=red!10,bottom color=blue!30}]
	\node[initial,accepting,state] (s0)                                 {$s_0$};
	\node[state,accepting] (s1) [right of = s0]                         {$s_1$};
	\node[state] (s2) [right of = s1]                         {$s_2$};
	\path [->] (s0) edge [loop above] node {$0$} (s0)
	           (s2) edge [loop above] node {$0$} (s2)
		   (s1) edge [bend left] node {$0$} (s0)
		   (s0) edge [bend left] node {$1$} (s1)
		   (s1) edge [bend left] node {$1$} (s2)
		   (s2) edge [bend left] node {$1$} (s1);
        \end{tikzpicture}
	\caption{Transition graph of the DFA $A$ in Example \ref{exam1_observability}.}
	\label{fig1:observability}
	\end{figure}



\end{example}

\begin{example}\cite[Example 2]{Kari2013LectureNoteonAFL}\label{exam11_observability}
	The DFA shown in Fig. \ref{fig14:observability} recognizes the regular language of
	words over the alphabet $\{a,b\}$ that contain the word $aa$.

	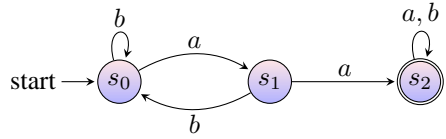
\begin{figure}
        \centering
\begin{tikzpicture}[>=stealth',shorten >=1pt,auto,node distance=2 cm, scale = 1, transform shape,
	->,>=stealth,inner sep=2pt,state/.style={shape=circle,draw,top color=red!10,bottom color=blue!30}]
	\node[initial,state] (s0)                                 {$s_0$};
	\node[state] (s1) [right of = s0]                         {$s_1$};
	\node[state,accepting] (s2) [right of = s1]                         {$s_2$};
	\path [->] (s0) edge [loop above] node {$b$} (s0)
	           (s2) edge [loop above] node {$a,b$} (s2)
		   (s1) edge [bend left] node {$b$} (s0)
		   (s0) edge [bend left] node {$a$} (s1)
		   (s1) edge node {$a$} (s2);
        \end{tikzpicture}
	\caption{Transition graph of the DFA $A$ in Example \ref{exam11_observability}.}
	\label{fig14:observability}
\end{figure}

\end{example}

Now we give a proposition on finite automata that will be frequently used in the sequel.

\begin{proposition}\cite{Zhang2014ObservabilityofBCN}\label{prop6_observability}
	Given a DFA $A=(S,\Sig,\s,s_0,F)$, assume that $F=S$ and for each
	$s$ in $S$, there is a word $u\in\Sig^*$ such that
	$\s(s_0,u)=s$. Then $L(A)=\Sig^*$ iff $A$ is complete.
\end{proposition}


\subsection{The semi-tensor product of matrices}

Since the framework of STP is used in this paper, some necessary concepts are introduced.

\begin{definition}\cite{Cheng(2011book)}
	Let $A\in \R_{m\times n}$, $B\in \R
	_{p\times q}$, and $\alpha=\mbox{lcm}
	(n,p)$ be the least common multiple of $n$ and $p$. The STP of $A$
	and $B$ is defined as
	\begin{equation*}
		A\ltimes B = \left(A\otimes I_{\frac{\alpha}{n}}\right)\left(B\otimes I_{\frac
		{\alpha}{p}}\right),
		\label{def_of_stp}
	\end{equation*}
	where $\otimes$ denotes the Kronecker product.
\end{definition}

From this definition, it follows that the conventional
product of matrices is a particular case of STP.
Since STP keeps many properties of the conventional product \cite{Cheng(2011book)},
e.g., the associative law, the distributive law,
etc.,
we usually omit the symbol ``$\ltimes$'' hereinafter.

\subsection{Boolean control networks and their algebraic forms}

In this paper, we investigate the following BCN
with $n$ state nodes, $m$ input nodes and $q$ output nodes:
\begin{equation}\label{BCN1}
\begin{split}
 &x (t + 1) = f (u (t),x (t)), \\
 &y(t)=h(x (t)),\\
 \end{split}
\end{equation}
where $t=0,1,\dots$, for each such $t$, $x(t)\in\D^n$, $u(t)\in\D^m$, $y(t)\in\D^q$,
$f:\D^{n+m}\to\D^n$ and $h:\D^n\to\D^q$ are Boolean mappings.

Using the STP of
matrices, (\ref{BCN1}) can be represented equivalently in
the following algebraic form \cite{Cheng2009bn_ControlObserva}:
\begin{equation}\label{BCN2}
\begin{split}
 &x(t + 1) = Lu(t)x(t),\\
 &y(t) = Hx(t),
\end{split}
\end{equation}
where $t=0,1,\dots$, for each such $t$, $x(t)\in\Delta_{N}$, $u(t)\in\Delta_{M}$ and $y(t)\in\Dt_{Q}$ denote
states, inputs and outputs, respectively,
$L\in\LM_{N\times (NM)}$, $H\in\LM_{Q\times N}$,
hereinafter, $N:=2^n$, $M:=2^m$ and $Q:=2^q$. In the framework of STP,
$N,M,Q$ can be any positive integers.
When $N,M,Q$ are not necessarily
powers of $2$, the corresponding network is called mix-valued (or finite-valued)
logical control network (cf. \cite{Cheng(2011book)}).
The details on the properties of STP, 
and how to transform a BCN into its equivalent algebraic form
can be found in \cite{Cheng(2011book)}.

\section{Preliminary results: using finite automata to determine reconstructibility}
\label{sec:DeterminingReconstructibility}

\subsection{Preliminary notations}

The input-state-output-time transfer graph
of the BCN \eqref{BCN2} is shown in Fig.
\ref{fig:input-state-output-time-graph}.

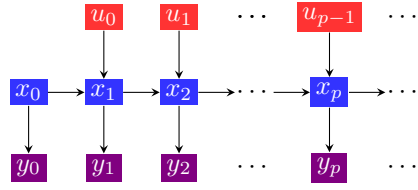
\begin{figure}
        \centering
        \begin{tikzpicture}[->,>=stealth,node distance=1.0cm]
          \tikzstyle{state}=[shape=rectangle,fill=blue!80,draw=none,text=white,inner sep=2pt]
          \tikzstyle{input}=[shape=rectangle,fill=red!80,draw=none,text=white,inner sep=2pt]
          \tikzstyle{output}=[shape=rectangle,fill=red!50!blue,draw=none,text=white,inner sep=2pt]
          \tikzstyle{disturbance}=[shape=rectangle,fill=red!20!blue,draw=none,text=white,inner sep=2pt]
	  \tikzstyle{time}=[inner sep=0pt,minimum size=5mm]
	  \node [state] (x0)              {$x_0$};
          \node [state] (x1) [right of=x0] {$x_1$};
	  \node [state] (x2) [right of=x1] {$x_2$};
	  \node [input] (u0) [above of=x1] {$u_0$};
	  \node [input] (u1) [right of=u0] {$u_1$};
	  \node [time] (x3) [right of=x2] {$\cdots$};
	  \node [time] (u2) [above of=x3] {$\cdots$};
	  \node [state] (x4) [right of=x3] {$x_p$};
	  \node [time] (u2) [above of=x3] {$\cdots$};
	  \node [input] (u3) [above of=x4] {$u_{p-1}$};
	  \node [time] (x5) [right of=x4] {$\cdots$};
	  \node [time] (u4) [above of=x5] {$\cdots$};
	  \node [output] (y0) [below of=x0] {$y_0$};
	  \node [output] (y1) [below of=x1] {$y_1$};
	  \node [output] (y2) [below of=x2] {$y_2$};
	  \node [output] (y4) [below of=x4] {$y_p$};
	  \node [time] (y3) [below of=x3] {$\cdots$};
	  \node [time] (y5) [below of=x5] {$\cdots$};
          \path (x0) edge (x1)
	        (x1) edge (x2)
		(u0) edge (x1)
		(u1) edge (x2)
		(x2) edge (x3)
		(x3) edge (x4)
		(u3) edge (x4)
		(x4) edge (x5)
		(x0) edge (y0)
		(x1) edge (y1)
		(x2) edge (y2)
		(x4) edge (y4);
        \end{tikzpicture}
	\caption{Input-state-output-time transfer graph of BCN \eqref{BCN2}, where subscripts denote time steps, $x_0,x_1,\dots$ denote states, $u_0,u_1,\dots$ denote inputs, $y_1,y_2,\dots$ denote outputs, and arrows infer dependence.}
	\label{fig:input-state-output-time-graph}
    \end{figure}

Now we define the following mappings from the set of input sequences to the set of state/output
sequences generated by the BCN \eqref{BCN2}.
Regarding $\Dt_N,\Dt_M,\Dt_Q$ as
alphabets, then finite input/state/output sequences can be seen as words in
$(\Dt_M)^*/(\Dt_N)^*/(\Dt_Q)^*$. For all $x_0\in\Dt_N$ and $p\in\N$, \begin{enumerate}
	\item
		if $p=0$,
		\begin{equation}
			\begin{split}
				L_{x_0}^p&:\{\e\}\to\{x_0\},\e\mapsto x_0,\\
				(HL)_{x_0}^p&:\{\e\}\to\{y_0\},\e\mapsto y_0,
			\end{split}
			\label{lx00}
		\end{equation}
		else
	\item
		\begin{equation}\label{lx0n}
			\begin{split}
				&L_{x_0}^p:(\Dt_M)^{p}\to(\Dt_N)^{p+1},u_0
				\dots u_{p-1}\mapsto x_0x_1\dots x_p,\\
				&(HL)_{x_0}^p:(\Dt_M)^{p}\to(\Dt_Q)^{p+1},u_0
				\dots u_{p-1}\mapsto y_0y_1\dots y_p,\\
			\end{split}
		\end{equation}
	\item
		\begin{equation}\label{lx0N}
			\begin{split}
				&L_{x_0}^{\N}:(\Dt_M)^{\N}\to(\Dt_N)^{\N},u_0u_1
				\dots \mapsto x_0x_1\dots,\\
				&(HL)_{x_0}^{\N}:(\Dt_M)^{\N}\to(\Dt_Q)^{\N},u_0u_1
				\dots \mapsto y_0y_1\dots .\\
			\end{split}
		\end{equation}
\end{enumerate}

\subsection{Reconstructibility of Boolean control networks}

In this subsection, we give the formal definition of reconstructibility.

\begin{definition}
	A BCN \eqref{BCN2} is called reconstructible, if there is an input sequence
	$U\in(\Dt_M)^{p}$	for some $p\in\N$ such that for all different $x_0,\bar x_0\in\Dt_N$,
	$L_{x_0}^{p}(U)[p]\ne L_{\bar x_0}^{p}(U)[p]$ implies $(HL)_{x_0}^{p}(U)\ne (HL)_
	{\bar x_0}^{p}(U)$.
	\label{def1_reconstructibility}
\end{definition}

Definition \ref{def1_reconstructibility} shows that if a BCN \eqref{BCN2} is reconstructible,
then there is an input sequence (called homing input sequence) such that no matter what the initial state the BCN is in,
one can use the corresponding output sequence to determine the current state.
Note that this reconstructibility is independent of the initial state, so any time can be seen as the initial
time.

From Definition \ref{def1_reconstructibility}, one immediately deduces that a BCN \eqref{BCN2}
is not reconstructible iff for all $p\in\N$ and input sequences $U\in(\Dt_M)^p$
there exist different states $x_0,\bar x_0\in\Dt_N$ such that,
$L_{x_0}^{p}(U)[p]\ne L_{\bar x_0}^{p}(U)[p]$ and
$(HL)_{x_0}^{p}(U)= (HL)_{\bar x_0}^{p}(U)$.

Furthermore,
the following proposition will play an important role in determining the reconstructibility of the BCN
\eqref{BCN2}.

\begin{proposition}\label{prop1_reconstructibility}
	A BCN \eqref{BCN2} is not reconstructible iff for all $p\in\N$
	and input sequences $U\in(\Dt_M)^p$ there exist different states
	$x_0,\bar x_0\in\Dt_N$ such that $L_{x_0}^{p}(U)[i]\ne L_{\bar x_0}^{p}(U)[i]$
	for any $i\in[0,p]$ and
	$(HL)_{x_0}^{p}(U)= (HL)_{\bar x_0}^{p}(U)$.
\end{proposition}

%
%
%

\subsection{Weighted pair graphs}\label{WPG}

In this subsection, we define a weighted directed graph for the BCN \eqref{BCN2},
named weighted pair graph, which is the key tool used for determining
its reconstructibility.


\begin{definition}\label{pairgraph_reconstructibility}
	Given a BCN \eqref{BCN2},
	a weighted directed graph $\Gr=(\V,\E,\W,2^{\Dt_M})$, where $\V$ denotes the vertex
	set, $\E\subset\V\times\V$ denotes the edge set, and
	$\W:\E\to2^{\Dt_M}$ denotes the weight function,
	is called the weighted pair graph of the BCN, if
	$\V = \{\{x,x'\}|x,x'\in\Dt_N,x\ne x',Hx=Hx'\}$;
			\footnote{In \cite{Zhang2014ObservabilityofBCN},
			$\V = \{(x,x')|x,x'\in\Dt_N,Hx=Hx'\}$.
			The unique difference between the weighted pair graph used to
			determine the observability of BCNs
			and the one defined here
			is that they have different vertex sets. Note that $\{x,x'\}$
			is an unordered pair, i.e., $\{x,x'\}=\{x',x\}$.}
			for all $(\{x_1,x_1'\},\{x_2,x_2'\})\in\V\times\V$,
			$(\{x_1,x_1'\},\{x_2,x_2'\})\in\E$ iff there exists $u\in\Dt_M$
	such that $Lux_1=x_2$ and $Lux_1'=x_2'$, or, $Lux_1=x_2'$ and $Lux_1'=x_2$;
	for all edges $e=(\{x_1,x_1'\},\{x_2,x_2'\})\in\E$,
	$\W(e)=\{u\in\Dt_M|Lux_1=x_2\text{ and }Lux_1'=x_2',\text{ or, }Lux_1=x_2'\text{ and }Lux_1'=x_2
	\}$.
\end{definition}

Similar to the transition graph of a DFA, we do not draw the curly
bracket ``$\{\}$'' to denote the weights of edges either.
Hereinafter, if $\Dt_M$ is known, we use $(\V,\E,\W)$ to denote $(\V,\E,\W,2^{\Dt_M})$
for short.
From Definition \ref{pairgraph_reconstructibility} it follows that the vertex set $\V$
consists of all pairs of different states that produce the same output.
There is an edge from vertex $v_1\in\V$ to vertex $v_2\in\V$ iff there is an
input driving one state in $v_1$ to one state in $v_2$ meanwhile driving
the other state in $v_1$ to the other state in $v_2$.
Later on, we will show that such a directed graph will be used to return a DFA that determines the
reconstructibility of the BCN \eqref{BCN2}.

\subsection{Determining reconstructibility}\label{sec3}

Next we design Algorithm \ref{alg5_reconstructibility}
that receives the weighted pair graph $\Gr$ of a BCN \eqref{BCN2}, and returns a DFA $A_{\V}$.
The DFA accepts exactly all finite
input sequences that cannot be used to determine the corresponding current states, i.e., the
non-homing input sequences.

Note that in Algorithm \ref{alg5_reconstructibility},  $\V$ is the initial state of
$A_{\V}$. Regarding $\V$ as a vertex of the transition graph of $A_{\V}$,
this algorithm first finds out all new children
of $\V$ and new transitions, second finds out all new children of all known children of $\V$ and new transitions,
and so on.
Since there are finitely many vertices, after a finite number of steps, the algorithm terminates.

In Algorithm \ref{alg5_reconstructibility},
$S$, $S^1_{\temp}$, $S^2_{\temp}$ can be seen as three containers consisting of some states of
$A_{\V}$. According to this, $s,s_j$ are states of $A_{\V}$, i.e.,
sets of vertices of $\Gr$; $v,v_s$ are vertices of $\Gr$.
At each step, the algorithm finds new
states of $A_{\V}$ that do not belong to $S$, puts them into $S^2_{\temp}$,
and at the end of each step, clears $S^1_{\temp}$,
puts the elements of $S^{2}_{\temp}$ into both $S$ and
$S^1_{\temp}$, and then clears $S^{2}_{\temp}$. The algorithm repeats this step until $S^{1}_{\temp}$ becomes empty,
i.e., all states of $A_{\V}$ have been found.

\begin{algorithm}
\caption{An algorithm for returning a DFA  that
determines the reconstructibility of the BCN \eqref{BCN2}}
\label{alg5_reconstructibility}
\begin{algorithmic}[1]
    \REQUIRE A BCN \eqref{BCN2} and its weighted pair graph $\Gr=(\V,\E,\W)$
    \ENSURE A DFA $A_{\V}$
	\STATE
		Let $S$, $S$, $\Dt_M$, $\s$ and $\V$ be the state set, the final state set, the alphabet, the partial transition
		function and the initial state of the DFA, respectively
	\STATE $S:=\{\V\}$, $S_{\temp}^{1}:=\{\V\}$, $S_{\temp}^{2}:=\emptyset$
	\WHILE{$S_{\temp}^{1}\ne\emptyset$}
	\FOR{all $s\in S_{\temp}^{1}$ and $j\in[1,M]$}
	\STATE $s_j:=\{v_s\in\V|\text{there is }v\in s\text{ such that }(v,v_s)\in\E\text{ and }
		\dt_M^j\in\W((v,v_s))\}$
		\IF{$s_j\ne\emptyset$ and $s_j\notin S$}
		\STATE $S:=S\cup \{s_j\}$, $S_{\temp}^{2}:=S_{\temp}^2\cup \{s_j\}$,
			$\s(s,\dt_M^j):=s_j$
			\ELSE\IF{$s_j\ne\emptyset$}
			\STATE $\s(s,\dt_M^j):=s_j$
			\ENDIF
		\ELSE
		\STATE $\s$ is not well defined at $(s,\dt_M^j)$
		\ENDIF
	\ENDFOR
	\STATE $S_{\temp}^1:=S_{\temp}^{2}$, $S_{\temp}^2:=\emptyset$
	\ENDWHILE
	\end{algorithmic}
\end{algorithm}

Based on Proposition \ref{prop1_reconstructibility} and Algorithm \ref{alg5_reconstructibility},
the following theorem holds.
\begin{theorem}
	A BCN \eqref{BCN2} is not reconstructible
	iff the DFA $A_{\V}$ generated by
	Algorithm \ref{alg5_reconstructibility} recognizes language $(\Dt_M)^*$, i.e.,
	$L(A_{\V})=(\Dt_M)^*$.
	\label{thm1_reconstructibility}
\end{theorem}

\begin{proof}
	``if'': Assume that $L(A_{\V})=(\Dt_M)^*$. Then for all $p\in\Z_{+}$
	and input sequences $U\in(\Dt_M)^p$, there are $p+1$ states of the DFA $A_{\V}$,
	denoted by $s_0,s_1,\dots,s_p$, such that $\s(s_i,U(i))=s_{i+1}$, $i\in[0,p-1]$,
	where $s_0=\V$, $\s$ is the partial transition function of $A_{\V}$.
	Furthermore, by Definition \ref{pairgraph_reconstructibility} and Algorithm \ref{alg5_reconstructibility},
	there are $p+1$ vertices $v_i=\{x_i,\bar x_i\}\in s_i$ of the weighted pair
	graph, 	$i\in[0,p]$, such that $x_i,\bar x_i\in\Dt_N$, $Hx_i=H\bar x_i$,
	$x_i\ne\bar x_i$, $L_{x_0}^p(U)=x_0x_1\dots x_p$, $L_{\bar x_0}^p(U)=\bar x_0
	\bar x_1\dots \bar x_p$, and $(HL)_{x_0}^p(U)=(HL)_{\bar x_0}^p(U)$.
	Besides, $\e\in L(A_{\V})$ implies that $\V\ne \emptyset$, i.e., there are different
	states of the BCN \eqref{BCN2} that produce the same output. Then by Proposition
	\ref{prop1_reconstructibility}, this BCN is not reconstructible.

	``only if'': Assume that $L(A_{\V})\subsetneq(\Dt_M)^*$, we prove that each input
	sequence in $(\Dt_M)^*\setminus L(A_{\V})$ determines the corresponding current state
	(i.e., each input sequence outside of $L(A_{\V})$ is a homing input sequence).
	Arbitrarily choose $U\in(\Dt_M)^*\setminus L(A_{\V})$.  If $U=\e$, then $\e\notin L(A_{\V})$, and
	$\V=\emptyset$. Hence $H$ is of full column rank, and the initial state can be determined by the initial output.
	Next we assume that $U\ne\e$. For all different states $x_0,\bar x_0\in\Dt_N$,
	$L_{x_0}^{p}(U)[p]\ne L_{\bar x_0}^{p}(U)[p]$ implies $(HL)_{x_0}^{p}(U)\ne (HL)_
	{\bar x_0}^{p}(U)$. Suppose on the contrary that there exist two different states
	$x_0',\bar x_0'\in\Dt_N$ such that
	$L_{x_0'}^{p}(U)[p]\ne L_{\bar x_0'}^{p}(U)[p]$ and $(HL)_{x_0'}^{p}(U)= (HL)_
	{\bar x_0'}^{p}(U)$. Then $U\in L(A_{\V})$ by Algorithm \ref{alg5_reconstructibility},
	which is a contradiction. Hence each $U$ outside of $L(A_{\V})$ is a homing input sequence,
	and the BCN is reconstructible.

\end{proof}

One can use Proposition \ref{prop6_observability} and
Theorem \ref{thm1_reconstructibility} to determine whether a given BCN \eqref{BCN2} is reconstructible
or not. The proof of Theorem \ref{thm1_reconstructibility} shows that every input sequence
$U\in(\Dt_M)^*\setminus L(A_{\V})$ can be used to determine the corresponding current state.

%
%

\subsection{Determining the current state}

We have shown how to determine whether a given BCN \eqref{BCN2} is reconstructible,
and how to find homing input sequences of reconstructible BCNs.
Next we show how to use a homing input sequence to determine the current state.
Actually, similar ideas have been widely used, e.g., to achieve the initial state of
an observable linear control system in the literature, and to diagnose fault occurrence of BCNs
in \cite[Algorithm 2]{Fornasini2015FaultDetectionBCN}.
	
Given a reconstructible BCN \eqref{BCN2}, where the initial state $x_0\in\Dt_N$ is given and unknown.
By Theorem \ref{thm1_reconstructibility}, the language recognized by
the DFA $A_{\V}$ generated by Algorithm
\ref{alg5_reconstructibility} is a proper subset of $(\Dt_M)^*$.
Then for each given homing input sequence
$U=u_0\dots u_{p-1}\in(\Dt_M)^*\setminus L(A_{\V})$ and the corresponding output sequence
$Y=y_0\dots y_p\in(\Dt_Q)^{p+1}$,
Algorithm \ref{alg3_reconstructibility} returns the current state.
$\X_0$ in Algorithm \ref{alg3_reconstructibility}
contains all possible states producing output $y_0$, and hence contains $x_0$.
At each time step $1\le t\le p$, $\X_t$
contains all states that are driven from initial state $x_0$
by input sequence $u_0\dots u_{t-1}$ and correspond to output sequence $y_0\dots y_{t}$.
Then by Theorem \ref{thm1_reconstructibility},
$\X_p$ is a singleton, and the unique element of $\X_p$ is the current state.
Particularly if $U=\e$, then $Y\in \Dt_Q$, and the current state (i.e., the initial state) is
$x_0=H^{-1}Y=H^{T}Y$.

\begin{algorithm}
	\caption{An algorithm for determining the current state of a
	reconstructible BCN \eqref{BCN2}}
	\label{alg3_reconstructibility}

	\begin{algorithmic}[1]
		\REQUIRE A reconstructible BCN \eqref{BCN2}, a homing input sequence $U
		=u_0\dots u_{p-1}\in (\Dt_M)^*\setminus L(A_{\V})$
		of length $p$, and the corresponding output sequence $Y=y_0\dots y_p\in(\Dt_Q)^{p+1}$

		\ENSURE The current state
		\IF{$U=\e$}
		\RETURN $H^{-1}Y$, stop
		\ELSE
		\STATE $\X_0:=\{x|x\in\Dt_N,Hx=y_0\}$
		\FOR{$i=0$; $i<p$; $i++$}
		\STATE $\X_{i+1}:=\emptyset$
		\FOR{all $x\in\X_i$}
		\IF{$(HL)_x^1(u_{i})[1]=y_{i+1}$}
		\STATE
		$\X_{i+1}:=\X_{i+1}\cup\{L_x^1(u_i)[1]\}$
		\ENDIF
		\ENDFOR
		\ENDFOR
		\RETURN the unique element of $\X_p$, stop
		\ENDIF
	\end{algorithmic}
\end{algorithm}

\subsection{An illustrative example}
\label{sec:Examples}

This example illustrates the weighted pair graph, how Algorithm \ref{alg5_reconstructibility} works, and
how to use a homing input sequence to determine the current state.

\begin{example}[\cite{Fornasini2013ObservabilityReconstructibilityofBCN}]\label{exam6_observability}
	Consider the following logical dynamical network:
	\begin{equation}
		\begin{split}
			x(t+1) &= \dt_5[1,4,3,5,4,2,3,3,4,4]u(t)x(t),\\
			y(t) &= \dt_2[1,1,2,1,2]x(t),
		\end{split}
		\label{eqn5_observability}
	\end{equation}
	where $t=0,1,\dots$, $x\in\Dt_5$, $y,u\in\Dt$.

	Note that although $5$ is not a power of $2$, all results of
	this paper are valid for it.
	The weighted pair graph of the logical control network \eqref{eqn5_observability} is shown
	in Fig. \ref{fig12:reconstructibility}. Putting the weighted pair graph
	into Algorithm \ref{alg5_reconstructibility},
	Algorithm \ref{alg5_reconstructibility} returns a DFA $A_{\V}$. The process of
	generating $A_{\V}$ is shwon in Fig. \ref{fig15:reconstructibility}.
	From Fig. \ref{fig15:reconstructibility} it follows that $A_{\V}$ is not complete, then
	by Proposition \ref{prop6_observability} and Theorem \ref{thm1_reconstructibility}, the
	logical control network \eqref{eqn5_observability} is reconstructible.

	Note that $\dt_2^1\dt_2^1\notin L(A_{\V})$ by Fig. \ref{fig15:reconstructibility},
	then $\dt_2^1\dt_2^1$ is a homing input sequence. Next we use $\dt_2^1\dt_2^1$ to determine the
	current state.
	Choose as an unknown initial state $x_0=\dt_5^2$. 
	Then the output sequence is $Y=y_0y_1y_2$,
	where $y_0=y_1=\dt_2^1,y_2=\dt_2^2$.
	According to Algorithm \ref{alg3_reconstructibility}, $\X_0=\{\dt_5^1,\dt_5^2,\dt_5^4\}$,
	$\X_1=\{\dt_5^1,\dt_5^4\}$, $\X_2=\{\dt_5^5\}$.
	Hence the current state is $\dt_5^5$.

	\begin{figure}
	        \centering
\begin{tikzpicture}[>=stealth',shorten >=1pt,auto,node distance=1.5 cm, scale = 1.0, transform shape,
	->,>=stealth,inner sep=2pt,state/.style={shape=circle,draw,top color=red!10,bottom color=blue!30}]

	\node[state] (12) {$12$};
	\node[state] (14) [right of = 12]  		  {$14$};
	\node[state] (24) [right of = 14]  		  {$24$};
	\node[state] (35) [right of = 24]  		  {$35$};

	
	\path [->]
		   (12) edge node {\tiny$1$} (14)
		   (14) edge node {\tiny$2$} (24)
	 	;

        \end{tikzpicture}
	\caption{Weighted pair graph of logical control network \eqref{eqn5_observability},
	where number $ij$ in each circle denotes state pair $\{\dt_5^i,
	\dt_5^j\}$, number $k$ beside each edge denotes
	weight $\{\dt_2^{k}\}$ of the corresponding edge.}
	\label{fig12:reconstructibility}
\end{figure}

	\begin{figure}
        \centering
\begin{tikzpicture}[>=stealth',shorten >=1pt,auto,node distance=1.7 cm, scale = 1.0, transform shape,
	->,>=stealth,inner sep=2pt,state/.style={shape=rectangle,draw,top color=red!10,bottom color=blue!30}]

	\node[initial,accepting,state] (12) {$\begin{array}{c}12\\14\\24\\35\end{array}$};
	\node[] (14) [right of = 12]  		  {};

	\node[initial,accepting,state] (12') [right of = 14] {$\begin{array}{c}12\\14\\24\\35\end{array}$};
	\node[accepting,state] (14') [above right of = 12']  		  {$14$};
	\node[accepting,state] (24') [below right of = 12']  		  {$24$};
	\node[] (14-35') [below right of = 14'] {};

	\node[initial,accepting,state] (12'') [right of = 14-35'] {$\begin{array}{c}12\\14\\24\\35\end{array}$};
	\node[accepting,state] (14'') [above right of = 12'']  		  {$14$};
	\node[accepting,state] (24'') [below right of = 12'']  		  {$24$};
	\node[] (14-35'') [right of = 14''] {};

	\path [->]
		   (12') edge node {\tiny$1$} (14')
		   (12') edge node {\tiny$2$} (24')

		   (12'') edge node {\tiny$1$} (14'')
		   (12'') edge node {\tiny$2$} (24'')
		   (14'') edge node {\tiny$2$} (24'')

		   	 	;

        \end{tikzpicture}
		\caption{Process of Algorithm \ref{alg5_reconstructibility} receiving
		the weighted pair graph of logical control network \eqref{eqn5_observability} and returning
		DFA  $A_{\V}$,
	where number $ij$ in each rectangle denotes state pair $\{\dt_5^i,
	\dt_5^j\}$, number $k$ beside each edge denotes
	weight $\{\dt_2^{k}\}$ of the corresponding edge.}
	\label{fig15:reconstructibility}
	\end{figure}
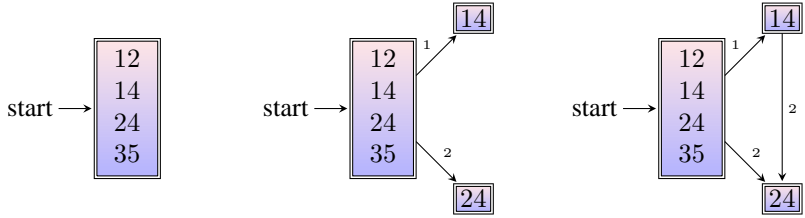

\end{example}

\section{Main results: using the weighted pair graph to determine reconstructibility}
\label{sec:EffectiveAlgorithmForReconstructibility}

In this section, we give more in-depth results on reconstructibility and the weighted pair graph
by using finite automata. And then based on these new results and the results in Section
\ref{sec:DeterminingReconstructibility},
we design a new algorithm (Algorithm \ref{alg4_reconstructibility})
for determining the reconstructibility of the BCN \eqref{BCN2}
directly from weighted pair graphs.
The new algorithm is significantly more efficient than the one
in Section \ref{sec:DeterminingReconstructibility}. After that, 
we analyze its computational complexity.

\subsection{How to directly use the weighted pair graph to determine reconstructibility}

Let us first look at the logical control network
\eqref{eqn5_observability} and its weighted pair graph shown in Fig. 
\ref{fig12:reconstructibility}.
One easily observes that in Fig. \ref{fig12:reconstructibility}
all vertices have outdegree less than $2$, i.e., the number of inputs of 
\eqref{eqn5_observability}, and in the rightmost graph shown in Fig. \ref{fig15:reconstructibility},
some vertex has outdegree less than $2$. On the contrary, one also observes that if for some BCN 
\eqref{BCN2}, each vertex of its weighted pair graph has outdegree $M$,
then in the transition graph of the corresponding DFA $A_{\V}$ generated by Algorithm
\ref{alg5_reconstructibility}, each vertex also has outdegree $M$.
Is it possible that establishing a {\it general theory} from these intuitive observations, in other words,
directly using the weighted pair graph to determine the completeness of the corresponding DFA 
$A_{\V}$ for every BCN  \eqref{BCN2}? If yes, the computational cost of determining the reconstructibility of the BCN 
\eqref{BCN2} will be significantly reduced. In this subsection, we give an affirmative answer to this problem.
The following proposition is a key result in answering the problem.

\begin{proposition}\label{prop3_reconstructibility}
	Given a weighted directed graph $\Gr=(\V,\E,\W,2^{\Sig})$, where the vertex set $\V$
	and the alphabet $\Sig$ are nonempty and finite,
	the edge set is $\E\subset \V\times\V$,
	and the weight function is $\W:\E\to 2^{\Sig}$. Assume that
	\begin{equation}
		\begin{split}
		&\text{for all }v\in\V \text{
		and }u_1,u_2\in\V,\text{ if }(v,u_1)\in\E,
		(v,u_2)\in\E\\&\text{and }u_1\ne u_2,\text{ then }
		\W( (v,u_1))\cap\W( (v,u_2))=\emptyset.
		\end{split}
		\label{eqn2_reconstructibility}
	\end{equation}
	Denote the set of vertices of $\Gr$ that have outdegree less than $|\Sig|$, or from which
	there is a path to a vertex having outdegree less than $|\Sig|$ by $\V_{|\Sig|}$.
	Substitute $\Gr$ into Algorithm \ref{alg5_reconstructibility} to generate
	a DFA $A_{\V}$. The following \eqref{item5_reconstructibility} and \eqref{item6_reconstructibility}
	hold.
	\begin{enumerate}[(i)]
		\item\label{item5_reconstructibility}
			$L(A_{\V})=\Sig^*$ iff graph $\Gr$ has a complete subgraph.
		\item\label{item6_reconstructibility}
			Graph $\Gr$ has a complete subgraph iff $\V_{|\Sig|}\subsetneq\V$.
	\end{enumerate}
\end{proposition}

\begin{proof} \eqref{item5_reconstructibility}
	``if'': Arbitrarily given a subgraph $\Gr'=(\V',\E',\W',2^{\Sig})$
	of graph $\Gr$, substitute $\Gr'$ into Algorithm \ref{alg5_reconstructibility}
	to generate a DFA $A_{\V'}$. Then $L(A_{\V'})\subset L(A_{\V})\subset \Sig^*$.

	Assume that in $\Gr'$,
	for all $v\in\V'$, $\outdeg(v)=|\Sig|$. Then in the transition graph of
	$A_{\V'}$, $\outdeg(\V')=|\Sig|$. Furthermore, all children of $\V'$ have outdegree
	$|\Sig|$, all children of all children of $\V'$ also have outdegree $|\Sig|$, and so on.
	Hence all vertices of the transition graph of $A_{\V'}$ have outdegree $|\Sig|$.
	That is, $A_{\V'}$ is complete by Proposition \ref{prop6_observability}.
	Hence $L(A_{\V'})=\Sig^*=L(A_{\V})$.
	
	``only if'': Assume that $L(A_{\V})=\Sig^*$. Let $A_{\V}$ be $(S,\Sig,\s,\V,S)$, where $S\subset 2^{\V}$.
	Then in the transition graph of $A_{\V}$,
	for all $s\in S$, $\outdeg(s)=|\Sig|$ by Proposition \ref{prop6_observability}.

	For any given state $\bar s\in S$, denote $\{s'\in S|\text{there is }u\in \Sig^*\text{
	such that }\s(\bar s,u)=s'\}$ by $S_{\bar s}$.
	Regarding the subgraph generated by $S_{\bar s}$ as a
	DFA $A_{\bar s}$, where $\bar s$ is the initial state,
	all elements of $S_{\bar s}$ are final states, then in the transition graph of $A_{\bar s}$,
	for all $s\in S_{\bar s}$, $\outdeg(s)=|\Sig|$.
	Then $L(A_{\bar s})=\Sig^*$ by Proposition \ref{prop6_observability}.

	Next we choose $\hat s\in S$ such that $|\hat s|=
	\min_{s\in S}|s|$, i.e., $\hat s$ has the minimal number of vertices of
	graph $\Gr$. Then by
	\eqref{eqn2_reconstructibility}, for all $s\in S_{\hat s}$, $|s|=|\hat s|$.
	Again by \eqref{eqn2_reconstructibility}, in $\Gr$, for all vertices
	$v\in\cup_{s\in S_{\hat s}}s$, $\outdeg(v)=|\Sig|$.
	Then in the subgraph of $\Gr$
	generated by $\cup_{s\in S_{\hat s}}s$, one also has that  for all vertices
	$v\in\cup_{s\in S_{\hat s}}s$, $\outdeg(v)=|\Sig|$. Hence the subgraph of $\Gr$
	generated by $\cup_{s\in S_{\hat s}}s$ is complete.

	\eqref{item6_reconstructibility} ``if'': Assume that $\V_{|\Sig|}\subsetneq\V$. We next
	we prove that subgraph $\Gr''$ of $\Gr$
	generated by $\V\setminus\V_{|\Sig|}$ is complete, that is, in $\Gr''$, each vertex
	has outdegree $|\Sig|$. Assume that in $\Gr''$,
	there is a vertex $v''\in\V\setminus\V_{|\Sig|}$ having
	outdegree less than $|\Sig|$. Since $v''\notin\V_{|\Sig|}$, $v''$ has outdegree $|\Sig|$
	in $\Gr$. Then $v''$ has a child in $\V_{|\Sig|}$, and then $v''\in\V_{|\Sig|}$, which is a
	contradiction.

	``only if'': Denote an arbitrary given complete subgraph of $\Gr$ by $\Gr''$.
	Then each vertex in $\Gr''$
	has outdegree $|\Sig|$ both in $\Gr$ and $\Gr''$, and all children of all vertices in
	$\Gr''$ belong to $\Gr''$. Hence $\V_{|\Sig|}\subsetneq\V$.
\end{proof}

Using Proposition \ref{prop3_reconstructibility}, Theorem \ref{thm1_reconstructibility}
can be directly simplified to the following Theorem \ref{thm7_reconstructibility}
which provides a significantly more effective algorithm for
determining reconstructibility than the one shown in Theorem \ref{thm1_reconstructibility}.
Theorem \ref{thm7_reconstructibility} is the main result of this paper.

\begin{theorem}\label{thm7_reconstructibility}
	A BCN \eqref{BCN2} is not reconstructible iff 
	its weighted pair graph has a complete subgraph.
\end{theorem}

Next we design Algorithm \ref{alg4_reconstructibility}
that provides an effective method for determining the reconstructibility of
the BCN \eqref{BCN2} based on Proposition \ref{prop3_reconstructibility} and
Theorem \ref{thm7_reconstructibility}.
The advantage of this algorithm with respect to Theorem \ref{thm1_reconstructibility}
is that it avoids constructing DFAs, so the computational cost is significantly reduced.
In Algorithm \ref{alg4_reconstructibility}, the set $\V_M$ can be found as follows:
\romannumeral1) find all vertices of $\Gr$ that have outdegree less than $M$, denote the set of
these vertices by $\V_0$, find all parents of the vertices of $\V_0$, where these parents are 
outside $\V_0$, the set of these parents is denoted by $\V_1$, remove $\V_0$ and all edges to vertices
of $\V_0$ from $\Gr$; \romannumeral2) find all parents of the vertices of $\V_1$, where these parents
are outside $\V_1$, the set of these parents is denoted by $\V_2$, remove $\V_1$ and all edges to vertices
of $\V_1$ from the remainder of $\Gr$; and so on. Finally we obtain the vertex set $\V_f=\V\setminus\V_M$,
then $\V_f\ne\emptyset$ iff $\V_M\subsetneq\V$. Note that when constructing
the weighted pair graph, we can directly obtain $\V_0$. Thus, Algorithm \ref{alg4_reconstructibility}
runs in time at most $|\V|$.

\begin{algorithm}
	\caption{An algorithm for determining the reconstructibility of a
	given BCN \eqref{BCN2}}
	\label{alg4_reconstructibility}

	\begin{algorithmic}[1]
		\REQUIRE A given BCN \eqref{BCN2} and its weighted pair graph $\Gr=(\V,\E,\W)$ 
		\ENSURE ``Yes'', if the BCN is reconstructible; ``No'', otherwise
		\STATE
		$\V_{M}:=\{v\in\V|\outdeg(v)<M,\text{ or there is a path from }v\text{
		to some vertex }v'\in\V\text{ satisfying }\outdeg(v')<M\}$
		\IF {$\V_{M}\subsetneq\V$}
		\RETURN ``No'', stop
		\ELSE
		\RETURN ``Yes'', stop
		\ENDIF
	\end{algorithmic}
\end{algorithm}

\begin{example}\label{exam3_reconstructibility}
	Recall the BCN \eqref{eqn5_observability}.
	Its weighted pair graph (shown in Fig. \ref{fig12:reconstructibility}) has no complete subgraph,
	then by Theorem \ref{thm7_reconstructibility}, the BCN is reconstructible.

\end{example}

\subsection{Computational complexity analysis}\label{subsec:complexityanalysis}

In this subsection,
we analyze the computational complexity of Theorem \ref{thm1_reconstructibility} and
Algorithm \ref{alg4_reconstructibility}.

Consider a BCN \eqref{BCN2} with $n$ state nodes and $m$ input nodes. Its weighted pair graph $\Gr$
has at most $((2^n)^2-2^n)/2=2^{2n-1}-2^{n-1}$ vertices, where the upper bound is tight and can be reached
(cf. Example \ref{exam6_reconstructibility}).
Since the DFA $A_{\V}$ in Theorem \ref{thm1_reconstructibility} is a power set construction of
$\Gr$, the size of $A_{\V}$ is bounded by $2^{2^{2n-1}-2^{n-1}}+2^{2^{2n-1}-2^{n-1}+m}$
(the number $2^{2^{2n-1}-2^{n-1}}$ of states plus the number $2^{2^{2n-1}-2^{n-1}+m}$ of transitions).
When constructing $A_{\V}$, one can obtain whether $A_{\V}$ is complete, hence the computational complexity
of using Theorem \ref{thm1_reconstructibility} to determine reconstructibility is $O(2^{2^{2n-1}+m})$.
The computational cost of constructing the weighted pair graph is bounded by 
$2^{2n-1}-2^{n-1}+(2^{2n-1}-2^{n-1})2^m$ (the number $2^{2n-1}-2^{n-1}$ of vertices plus the number
$(2^{2n-1}-2^{n-1})2^m$ of edges), and Algorithm \ref{alg4_reconstructibility} runs in time at most
$|\V|\le 2^{2n-1}-2^{n-1}$, hence the overall computational cost of using Algorithm \ref{alg4_reconstructibility}
to determine reconstructibility is bounded by $2^{2n-1}-2^{n-1}+(2^{2n-1}-2^{n-1})2^m+2^{2n-1}-2^{n-1}=2^{2n}-2^n+
(2^{2n-1}-2^{n-1})2^m$. Hence Algorithm \ref{alg4_reconstructibility} runs in time $O(2^{2n}+2^{2n-1+m})$,
which is significantly more efficient than Theorem \ref{thm1_reconstructibility}.

Does there exist a more effective algorithm for determining reconstructibility?
Next we show that the problem of determining reconstructibility
is {\bf NP}-hard, which means there exists no polynomial time algorithm for determining reconstructibility 
unless {\bf P}={\bf NP}, i.e., it is very unlikely that there exists no 
polynomial time algorithm for determining reconstructibility.

In \cite{Laschov2013ObservabilityofBN:GraphApproach}, 
it is proved that the problem of determining the observability of BNs/BCNs is {\bf NP}-hard
by reducing the well-known {\bf NP}-complete Boolean satisfiability problem to the problem of determining 
the observability of BNs/BCNs.
This reduction can be directly used to prove the {\bf NP}-hardness of the problem of determining
the reconstructibility of BNs/BCNs.

The Boolean satisfiability problem is stated as below: given a Boolean function with $n$ arguments $g:\D^n\to \D$,
determine whether $g$ is satisfiable, i.e., whether there exist 
$x_1^*,\dots,x_n^*\in\D$ such that $g(x_1^*,\dots,x_n^*)=1$.

Consider the following BN with $n$ state nodes and $q$ output nodes:
\begin{equation}
	\begin{split}
		&x(t+1) = f(x(t)),\\
		&y(t) = h(x(t)),
	\end{split}
	\label{BN1}
\end{equation}
where $t=0,1,\dots,$ for each such $t$, $x(t)\in\D^n$, $y(t)\in\D^q$, $f:\D^n\to \D^n$ and $h:\D^n\to\D^q$
are Boolean mappings.

A BN \eqref{BN1} is called observable \cite{Laschov2013ObservabilityofBN:GraphApproach}, if for every two different
initial states $(x_1^*,\dots,x_n^*)$ and $(x_1^{**},\dots,x_n^{**})$ both in $\D^n$, the corresponding
output sequences are different.

Here we call a BN \eqref{BN1} reconstructible, if there exists $p\in\N$ such that for every two different
initial states $(x_1^*,\dots,x_n^*)$ and $(x_1^{**},\dots,x_n^{**})$ both in $\D^n$,
if the corresponding states at time step $p$ are different, then the corresponding
output sequences are different.

	Construct a set of BNs as follows:
	\begin{equation}
		\begin{split}
			&x_k(t+1) = x_k(t)\oplus (x_{k+1}(t)\odot \cdots\odot x_{n}(t))=:f_k(x_1(t),\dots,x_n(t)),\quad k\in[1,n-1],\\
			&x_n(t+1) = x_n(t)\oplus 1=:f_n(x_1(t),\dots,x_n(t)),\\
			&y(t) = x_1(t)\odot g(x_2(t),\dots,x_n(t)),
		\end{split}
		\label{eqn:SATtoReconstructBN}
	\end{equation}
	where $t=0,1,\dots$, for each such $t$,
	$x_1(t),\dots,x_n(t)\in\D$, $\oplus$ and $\odot$ denote the  addition and multiplication modulo $2$, 
	respectively, and $f_1,\dots,f_n:\D^n\to\D$ and $g:\D^{n-1} \to \D$ are Boolean functions.
	It can be seen that the state transition graph of the BN \eqref{eqn:SATtoReconstructBN} 
	is a simple cycle of length $2^n$. For example, when $n = 3$, the state transition graph is 
	$000 \rightarrow 001 \rightarrow  010 \rightarrow 011 \rightarrow100 \rightarrow101 \rightarrow110 \rightarrow111 \rightarrow000$. Hence the BN \eqref{eqn:SATtoReconstructBN} is observable iff it is reconstructible.

It is proved in \cite[Proposition 10]{Laschov2013ObservabilityofBN:GraphApproach} that for the BN
\eqref{eqn:SATtoReconstructBN}, $g$ is satisfiable iff the BN is observable. 
Based on this property and the polynomial time construction of \eqref{eqn:SATtoReconstructBN} from $g$,
the problem of determining the observability of the BN \eqref{eqn:SATtoReconstructBN} is {\bf NP}-hard.
Then by the equivalence of the observability and reconstructibility of the BN \eqref{eqn:SATtoReconstructBN},
the problem of determining the reconstructibility of the BN \eqref{eqn:SATtoReconstructBN} is also {\bf NP}-hard.
As a corollary, the problem of determining the reconstructibility of the BCN \eqref{BCN1} is also
{\bf NP}-hard by reducing the reconstructibility of the BN \eqref{eqn:SATtoReconstructBN} to that of the BCN \eqref{BCN1}.

The following Example \ref{exam6_reconstructibility} shows that for a given BCN \eqref{BCN2},
the upper bound $2^{2n-1}-2^{n-1}$ on the number of vertices of its weighted pair graph can be
reached.

\begin{example}\label{exam6_reconstructibility}
	Consider a set of logical control networks
\begin{equation}
	\begin{split}
		x(t+1) &= [L_1,\dots,L_N]x(t)u(t),\\
		y(t) &= \dt_N^1,
	\end{split}
	\label{eqn6_reconstructibility}
\end{equation}
where $t=0,1,\dots$, for each such $t$, $x(t),y(t),u(t)\in\Dt_N$, $N$ is an integer no less than $3$,
$L_i\in\LM_{N\times N}$, $i\in[1,N]$, for all $i,j\in[1,N]$,
\begin{equation}
	\col_{j}(L_i)=\left\{
	\begin{array}[]{ll}
		\dt_N^i, &\text{ if }j\ne i,\\
		\dt_N^{i \mod N+1}, &\text{ if }j=i.\\
	\end{array}\right.
\end{equation}

For each integer $N$ no less than $3$, the weighted pair graph $\Gr_N=(\V_N,\E_N,\W_N,2^{\Dt_N})$ 
of the logical control network \eqref{eqn6_reconstructibility} satisfies the following properties:

\begin{enumerate}[(i)]
	\item\label{item1_reconstructibility_worst} $\Gr_N$ has
		exactly $(N^{2}-N)/2$ vertices and $\V_N=\{\{\dt_N^i,\dt_N^j\}|
		i,j\in[1,N], i\ne j\}$, because the output function is constant.
	\item\label{item2_reconstructibility_worst}
Each vertex has outdegree $2$.
	\item\label{item3_reconstructibility_worst}
For all different $i,j\in[1,N]$, $\W_N((\{\dt_N^i,\dt_N^j\},\{\dt_N^i,\dt_N^{i\mod N+1}\}))
=\{\dt_N^i\}$. 
	\item\label{item4_reconstructibility_worst}
For each $i\in[1,N]$, vertex $\{\dt_N^{i},\dt_N^{i\mod N+1}\}$ has a self-loop with
weight $\{\dt_N^i\}$.
	\item\label{item5_reconstructibility_worst}
$\{\dt_N^1,\dt_N^2\}\xrightarrow[]{\{\dt_N^2\}}\{\dt_N^2,\dt_N^3\}\xrightarrow[]{\{\dt_N^3\}}
\cdots\xrightarrow[]{\{\dt_N^N\}}\{\dt_N^N,\dt_N^1\}\xrightarrow[]{\{\dt_N^1\}}\{\dt_N^1,\dt_N^2\}$
is the unique cycle that is not a self-loop.
\end{enumerate}

By \eqref{item2_reconstructibility_worst} we have none of these weighted pair graphs has a complete subgraph.
Then by Theorem \ref{thm7_reconstructibility}, all these logical control networks
are reconstructible, and the overall
computational cost is $(N^2-N)/2+(N^2-N)N/2$.

When $N=4$, the weighed pair graph is shown in Fig. \ref{fig2:reconstructibility}.

	\begin{figure}
        \centering
\begin{tikzpicture}[>=stealth',shorten >=1pt,auto,node distance=1.7 cm, scale = 1.0, transform shape,
	->,>=stealth,inner sep=2pt,state/.style={shape=circle,draw,top color=red!10,bottom color=blue!30}]

	\node[state] (13) {$13$};
	\node[state] (12) [above left of = 13]  		  {$12$};
	\node[state] (41) [below left of = 13]  		  {$41$};
	\node[state] (23) [above right of = 13]         {$23$};
	\node[state] (34) [below right of = 13]         {$34$};
	\node[state] (24) [below right of = 34]         {$24$};

	\path [->]
	(12) edge [loop above] node {\tiny$1$} (12)
	(12) edge node {\tiny$2$} (23)
	(23) edge [loop above] node {\tiny$2$} (23)
	(23) edge node {\tiny$3$} (34)
	(34) edge [loop right] node {\tiny$3$} (34)
	(34) edge node {\tiny$4$} (41)
	(41) edge [loop left] node {\tiny$4$} (41)
	(41) edge node {\tiny$1$} (12)
	(13) edge node {\tiny$1$} (12)
	(13) edge node {\tiny$3$} (34)
	(24) edge [bend right] node {\tiny$2$} (23)
	(24) edge [bend left] node {\tiny$4$} (41)
	 	;

        \end{tikzpicture}
		\caption{Weighted pair graph of logical control network \eqref{eqn6_reconstructibility}
		with $N$ equal to $4$,
	where number $ij$ in each circle denotes state pair $\{\dt_4^i,
	\dt_4^j\}$, number $k$ beside each edge denotes
	weight $\{\dt_4^{k}\}$ of the corresponding edge.}
	\label{fig2:reconstructibility}
\end{figure}
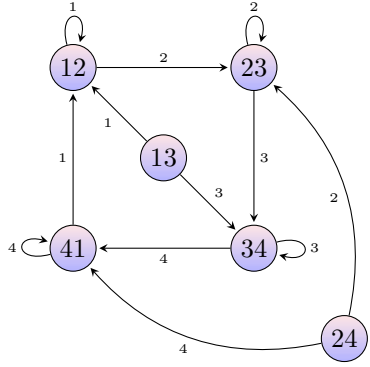

For each integer $N$ no less than $3$, put $\Gr_N$ into Algorithm \ref{alg5_reconstructibility} to obtain the DFA
$A_{\V_N}=(S_N,\Dt_N,\s_N,\V_N,S_N)$, where $S_N=\{\V_N\}\cup \bigcup_{k=1}^{N}\{\{\{\dt_N^{k},\dt_N^{k\mod N+1}\}\}\}$,
for all $k\in[1,N]$,
\begin{equation*}\begin{split}
&\s_N(\V_N,\dt_N^k)\\
&=\{\{\dt_N^k,\dt_N^{k\mod N+1}\}\},\\
&\s_N(\{\{\dt_N^k,\dt_N^{k\mod N+1}\}\},\dt_N^k)\\
&=\{\{\dt_N^k,\dt_N^{k\mod N+1}\}\},\\
&\s_N(\{\{\dt_N^k,\dt_N^{k\mod N+1}\}\},\dt_N^{k\mod N+1})\\
&=\{\{\dt_N^{k\mod N+1},\dt_N^{(k\mod N+1)\mod N+1}\}\},
\end{split}\end{equation*}
$\s_N$ is not well defined at any other element of $S_N\times \Dt_N$.
Hence in each DFA $A_{\V_N}$,  initial state $\V_N$ has outdegree $N$, and all other states
have outdegree $2$. That is, none of these DFAs is complete.
Then by Theorem \ref{thm1_reconstructibility} and Proposition \ref{prop6_observability},
we also have that all of these logical control networks are reconstructible.

Here we also choose $N=4$, the DFA with respect to the weighed pair graph generated by
Algorithm \ref{alg5_reconstructibility} is depicted in Fig. \ref{fig4:reconstructibility}.

	\begin{figure}
        \centering
\begin{tikzpicture}[>=stealth',shorten >=1pt,auto,node distance=3.5 cm, scale = 1.0, transform shape,
	->,>=stealth,inner sep=2pt,state/.style={shape=rectangle,draw,top color=red!10,bottom color=blue!30},
	point/.style={circle,inner sep=0pt,minimum size=2pt,fill=},
	skip loop/.style={to path={-- ++(0,#1) -| (\tikztotarget)}}]

	\node[initial,accepting,state] (13) {$\begin{array}{c}12,23,34\\41,13,24\end{array}$};
	\node[accepting,state] (12) [above left of = 13]  		  {$12$};
	\node[accepting,state] (41) [below left of = 13]  		  {$41$};
	\node[accepting,state] (23) [above right of = 13]         {$23$};
	\node[accepting,state] (34) [below right of = 13]         {$34$};

	\path [->]
	(12) edge [loop above] node {\tiny$1$} (12)
	(12) edge node {\tiny$2$} (23)
	(23) edge [loop above] node {\tiny$2$} (23)
	(23) edge node {\tiny$3$} (34)
	(34) edge [loop right] node {\tiny$3$} (34)
	(34) edge node {\tiny$4$} (41)
	(41) edge [loop left] node {\tiny$4$} (41)
	(41) edge node {\tiny$1$} (12)
	(13) edge node {\tiny$1$} (12)
	(13) edge node {\tiny$3$} (34)
	(13) edge node {\tiny$2$} (23)
	(13) edge node {\tiny$4$} (41)

	 	;

        \end{tikzpicture}
		\caption{DFA with respect to the weighted pair graph of logical control network
		\eqref{eqn6_reconstructibility}
		with $N$ equal to $4$ generated by Algorithm \ref{alg5_reconstructibility},
	where number $ij$ in each rectangle denotes state pair $\{\dt_4^i,
	\dt_4^j\}$, number $k$ beside each edge denotes
	weight $\{\dt_4^{k}\}$ of the corresponding edge.}
	\label{fig4:reconstructibility}
\end{figure}
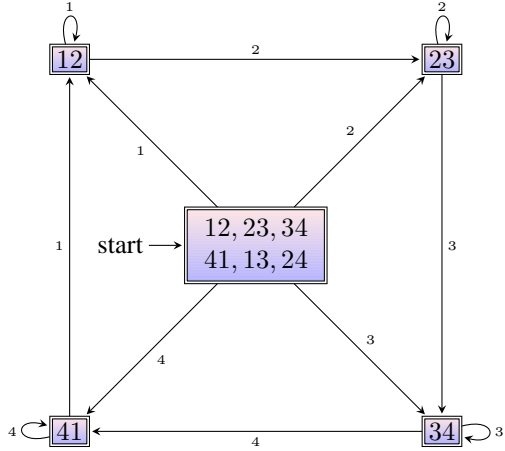

\end{example}

Note that each DFA $A_{\V}$ in Example \ref{exam6_reconstructibility} has exactly $N+1$ states,
which is less than the number $(N^2-N)/2$ of vertices of the weighted pair graph $\Gr_N$.
In this case using Algorithm \ref{alg4_reconstructibility} to determine the reconstructibility
of the logical control networks \eqref{eqn6_reconstructibility} is not much more effective than 
using Theorem  \ref{thm1_reconstructibility} and Proposition \ref{prop6_observability}.
Next we give one more example to show that Algorithm \ref{alg4_reconstructibility}
does perform significantly more effecient than Theorem \ref{thm1_reconstructibility}.

\begin{example}\label{exam2_reconstructibility}
	Consider a set of logical control networks
\begin{equation}
	\begin{split}
		x(t+1) &= [L_1,\dots,L_N]u(t)x(t),\\
		y(t) &= \dt_N^1,
	\end{split}
	\label{eqn1_reconstructibility}
\end{equation}
where $t=0,1,\dots$, for each such $t$, $x(t),y(t),u(t)\in\Dt_N$, $N$ is again an integer no less than $3$,
$L_i\in\LM_{N\times N}$, $i\in[1,N]$, for all $i,j\in[1,N]$,
\begin{equation}
	\col_{j}(L_i)=\left\{
	\begin{array}[]{ll}
		\dt_N^i, &\text{ if }j\ne i,\\
		\dt_N^{i \mod N+1}, &\text{ if }j=i.\\
	\end{array}\right.
\end{equation}

For each integer $N$ no less than $3$, the weighted pair graphs $\Gr_N=(\V_N,\E_N,\W_N,2^{\Dt_N})$ of the
logical control network \eqref{eqn1_reconstructibility}
satisfies the following properties:

\begin{enumerate}[(i)]
	\item\label{item8_reconstructibility_worst} $\Gr_N$ has
		exactly $(N^{2}-N)/2$ vertices and $\V_N=\{\{\dt_N^i,\dt_N^j\}|
		i,j\in[1,N], i\ne j\}$, because the output function is constant.
	\item\label{item9_reconstructibility_worst}
		Vertex $\{\dt_N^i,\dt_N^{i\mod N+1}\}$ has outdegree $N-1$ for each $i\in[1,N]$.
		All other vertices have outdegree $N$.
	\item\label{item10_reconstructibility_worst}
		For all different $i,j\in[1,N]$, $\W_N((\{\dt_N^i,\dt_N^j\},\{\dt_N^i,\dt_N^{j}\}))
		=\{\dt_N^k|k\in[1,N],k\ne i,j\}$.
		For all $i\in[1,N]$, $\W_N((\{\dt_N^i,\dt_N^{i\mod N+1}\},\{\dt_N^i,\dt_N^{
		(i\mod N+1)\mod N+1}\})) =\{\dt_N^{i\mod N+1}\}$.
		For all different $i,j\in[1,N]$ satisfying that $i\ne j\mod N+1$ and $j\ne i\mod N+1$,
		$\W_N((\{\dt_N^i,\dt_N^{j}\},\{\dt_N^{i\mod N+1},\dt_N^{j}\})) =\{\dt_N^{i}\}$.
	\item\label{item11_reconstructibility_worst}
		For all vertices $\{\dt_N^i,\dt_N^j\}$ satisfying that $i<j$ and $j\ge 3$, there is a path
		from $\{\dt_N^1,\dt_N^2\}$ to $\{\dt_N^i,\dt_N^j\}$:
		$\{\dt_N^1,\dt_N^2\}\xrightarrow[]{\{\dt_N^2\}}\{\dt_N^1,\dt_N^3\}\xrightarrow[]{\{\dt_N^3\}}
		\cdots\xrightarrow[]{\{\dt_N^{j-1}\}}\{\dt_N^1,\dt_N^j\}
		\xrightarrow[]{\{\dt_N^1\}}\{\dt_N^2,\dt_N^j\}\xrightarrow[]{\{\dt_N^2\}}
		\cdots\xrightarrow[]{\{\dt_N^{i-1}\}}\{\dt_N^i,\dt_N^j\}$.
	\item\label{item12_reconstructibility_worst}
		For all vertices $\{\dt_N^i,\dt_N^j\}$ satisfying that $i<j$ and $j\ge 3$, there is a path
		from $\{\dt_N^i,\dt_N^j\}$ to $\{\dt_N^1,\dt_N^2\}$:
		if $i=1$, the path is $\{\dt_N^1,\dt_N^j\}\xrightarrow[]{\{\dt_N^j\}}\{\dt_N^1,\dt_N^{j+1}\}
		\xrightarrow[]{\{\dt_N^{j+1}\}}\cdots\xrightarrow[]{\{\dt_N^{N-1}\}}\{\dt_N^1,\dt_N^{N}\}
		\xrightarrow[]{\{\dt_N^{1}\}}\{\dt_N^2,\dt_N^{N}\}\xrightarrow[]{\{\dt_N^{N}\}}\{\dt_N^2,\dt_N^{1}\}
		=\{\dt_N^1,\dt_N^{2}\}$; else, the path is $\{\dt_N^i,\dt_N^j\}\xrightarrow[]{\{\dt_N^j\}}
		\{\dt_N^i,\dt_N^{j+1}\}
		\xrightarrow[]{\{\dt_N^{j+1}\}}\cdots\xrightarrow[]{\{\dt_N^{N-1}\}}\{\dt_N^i,\dt_N^{N}\}
		\xrightarrow[]{\{\dt_N^{N}\}}\{\dt_N^i,\dt_N^{1}\}
		\xrightarrow[]{\{\dt_N^{i}\}}\{\dt_N^{i+1},\dt_N^{1}\}\xrightarrow[]{\{\dt_N^{i+1}\}}\cdots
		\xrightarrow[]{\{\dt_N^{N-1}\}}\{\dt_N^{N},\dt_N^{1}\}
		\xrightarrow[]{\{\dt_N^{1}\}}\{\dt_N^{N},\dt_N^{2}\}
		\xrightarrow[]{\{\dt_N^{N}\}}\{\dt_N^{1},\dt_N^{2}\}$.
	\item\label{item13_reconstructibility_worst}
		By \eqref{item10_reconstructibility_worst}, \eqref{item11_reconstructibility_worst}
		and \eqref{item12_reconstructibility_worst}, each graph $\Gr_N$ is strongly connected.
\end{enumerate}

Since $\outdeg(\{\dt_N^1,\dt_N^2\})=N-1<N$ and these weighted pair graphs are strongly connected,
by Algorithm \ref{alg4_reconstructibility}, these logical control networks \eqref{eqn1_reconstructibility}
are all reconstructible, and the overall computational cost is $(N^2-N)/2+(N^2-N)N/2$.

When $N=4$, the weighed pair graph is shown in Fig. \ref{fig9:reconstructibility}.

	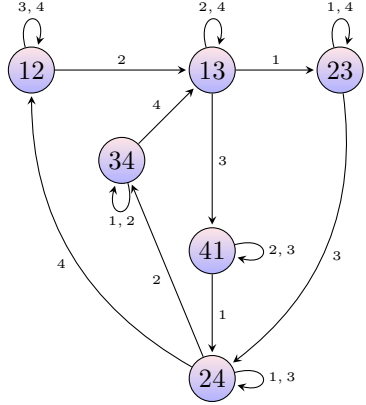
\begin{figure}
        \centering
\begin{tikzpicture}[>=stealth',shorten >=1pt,auto,node distance=1.7 cm, scale = 1.0, transform shape,
	->,>=stealth,inner sep=2pt,state/.style={shape=circle,draw,top color=red!10,bottom color=blue!30}]

	\node[state] (12) {$12$};
	\node[state] (34) [below right of = 12]  		  {$34$};
	\node[state] (13) [above right of = 34]  		  {$13$};
	\node[state] (23) [right of = 13]         {$23$};
	\node[state] (14) [below right of = 34]         {$41$};
	\node[state] (24) [below of = 14]         {$24$};

	\path [->]
	(12) edge [loop above] node {\tiny$3,4$} (12)
	(13) edge [loop above] node {\tiny$2,4$} (13)
	(23) edge [loop above] node {\tiny$1,4$} (23)
	(14) edge [loop right] node {\tiny$2,3$} (14)
	(34) edge [loop below] node {\tiny$1,2$} (34)
	(24) edge [loop right] node {\tiny$1,3$} (24)

	(12) edge node {\tiny$2$} (13)
	(13) edge node {\tiny$1$} (23)
	(13) edge node {\tiny$3$} (14)
	(23) edge [bend left] node {\tiny$3$} (24)
	(34) edge node {\tiny$4$} (13)
	(24) edge [bend left] node {\tiny$4$} (12)
	(14) edge node {\tiny$1$} (24)
	(24) edge node {\tiny$2$} (34)

	 	;

        \end{tikzpicture}
		\caption{Weighted pair graph of logical control network \eqref{eqn1_reconstructibility}
		with $N$ equal to $4$,
	where number $ij$ in each circle denotes state pair $\{\dt_4^i,
	\dt_4^j\}$, number $k$ beside each edge denotes
	weight $\{\dt_4^{k}\}$ of the corresponding edge.}
	\label{fig9:reconstructibility}
\end{figure}

For each integer $N$ no less than $3$, put graph $\Gr_N$ into Algorithm \ref{alg5_reconstructibility} to obtain the DFA
$A_{\V_N}=(S_N,\Dt_N,\s_N,\V_N,S_N)$, where $|S_N|=C_N^0+C_N^1+\cdots+C_N^{N-2}=2^N-N-1$,
where each $C_N^i$ denotes the number of states of $A_{\V_N}$ with cardinality $C_{N-i}^2$, e.g., when $N=4$
and $i=0$, the unique state of $A_{\V_4}$ with cardinality $C_{4-0}^2=6$ is $\V_4$. Since
$|\{ \mathcal{S}\subset[1,N]||\mathcal{S}|>1\}=:{\mathcal I}|=2^N-N-1=|S_N|$, 
states of $A_{\V_N}$ can be indexed by the set ${\mathcal I}$. For each ${\cal S\in I}$,
denote $s_{\cal S}:=\{\{\dt_N^i,\dt_N^j\}|i,j\in{\cal S},i\ne j\}$. Then $S_N=\{s_{\cal S}|\cal S\in I\}$, 
$\V_N=s_{[1,N]}\in{\cal S}_N$. In the transition graph of the DFA $A_{\V_N}$, each state $s_{\cal S}$ satisfying 
$|{\cal S}|>2$ has outdegree $N$, each state $s_{\cal S}$ satisfying $|{\cal S}|=2$ has outdegree $N$ or $N-1$. For
each state $\cal S\in I$, the children of $s_{\cal S}$ have cardinality $|\cal S|$ or $|{\cal S}|-1$. 
Since there is a state $s_{\cal S}$ satisfying $|{\cal S}|=2$ and $\outdeg(s_{\cal S})=N-1<N$, the DFA $A_{\V_N}$ is not
complete. 
Then by Theorem \ref{thm1_reconstructibility}  and Proposition \ref{prop6_observability},
all these logical control networks are reconstructible, and the overall computational cost is
$2^N-N-1+(2^N-N-1)N$. Hence in this case Algorithm \ref{alg4_reconstructibility} performs significantly 
more effective than Theorem \ref{thm1_reconstructibility}.

Choose $N=4$, the DFA with respect to the weighed pair graph generated by
Algorithm \ref{alg5_reconstructibility} is drawn in Fig. \ref{fig6:reconstructibility}.

	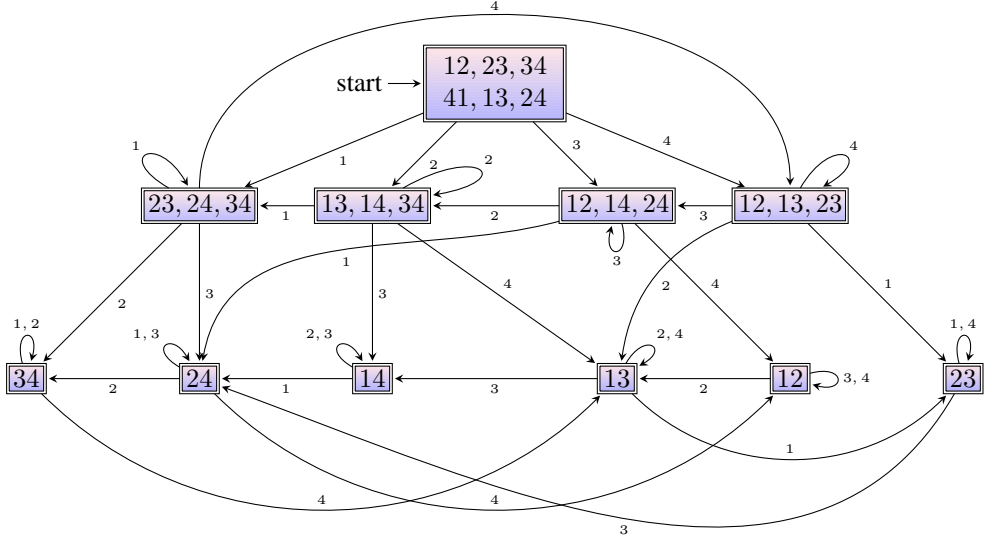
\begin{figure}
        \centering
\begin{tikzpicture}[>=stealth',shorten >=1pt,auto,node distance=2.3 cm, scale = 1.0, transform shape,
	->,>=stealth,inner sep=2pt,state/.style={shape=rectangle,draw,top color=red!10,bottom color=blue!30},
	point/.style={circle,inner sep=0pt,minimum size=2pt,fill=},
	skip loop/.style={to path={-- ++(0,#1) -| (\tikztotarget)}}]

	\node[initial,accepting,state]  (1234)  {$\begin{array}{c}12,23,34\\41,13,24\end{array}$};
	\node[accepting,state] (134) [below left of = 1234]  		  {$13,14,34$};
	\node[accepting,state] (124) [below right of = 1234]  		  {$12,14,24$};
	\node[accepting,state] (234) [left of = 134]  		  {$23,24,34$};
	\node[accepting,state] (123) [right of = 124]  		  {$12,13,23$};
	\node[accepting,state] (24) [below of = 234] {$24$};
	\node[accepting,state] (14) [below of = 134] {$14$};
	\node[accepting,state] (13) [below of = 124] {$13$};
	\node[accepting,state] (12) [below of = 123] {$12$};
	\node[accepting,state] (23) [right of = 12] {$23$};
	\node[accepting,state] (34) [left of = 24] {$34$};

	\path [->]
	(1234) edge node {\tiny$1$} (234)
	(1234) edge node {\tiny$2$} (134)
	(1234) edge node {\tiny$3$} (124)
	(1234) edge node {\tiny$4$} (123)

	(123) edge [out = 60, in = 30, loop] node {\tiny$4$} (123)
	(124) edge [loop below] node {\tiny$3$} (124)
	(134) edge [out = 30, in = 10, loop] node {\tiny$2$} (134)
	(234) edge [in = 120, out = 150, loop] node {\tiny$1$} (234)
	(234) edge [out = 90, in = 90] node {\tiny$4$} (123)

	(134) edge node {\tiny$1$} (234)
	(124) edge node {\tiny$2$} (134)
	(123) edge node {\tiny$3$} (124)

	(234) edge node {\tiny$2$} (34)
	(234) edge node {\tiny$3$} (24)
	(134) edge node {\tiny$3$} (14)
	(134) edge node {\tiny$4$} (13)
	(124) edge node {\tiny$4$} (12)
	(124) edge [out = -165, in = 80] node {\tiny$1$} (24)
	(123) edge node {\tiny$1$} (23)
	(123) edge [bend right] node {\tiny$2$} (13)

	(34) edge [loop above] node {\tiny$1,2$} (34)
	(24) edge [out = 150, in = 120, loop] node {\tiny$1,3$} (24)
	(14) edge [out = 150, in = 120, loop] node {\tiny$2,3$} (14)
	(13) edge [out = 60, in = 30, loop] node {\tiny$2,4$} (13)
	(12) edge [loop right] node {\tiny$3,4$} (12)
	(23) edge [loop above] node {\tiny$1,4$} (23)

	(24) edge node {\tiny$2$} (34)
	(14) edge node {\tiny$1$} (24)
	(13) edge node {\tiny$3$} (14)
	(12) edge node {\tiny$2$} (13)

	(34) edge [out = -45, in = -135] node {\tiny$4$} (13)
	(24) edge [out = -45, in = -135] node {\tiny$4$} (12)
	(13) edge [out = -45, in = -135] node {\tiny$1$} (23)
	(23) edge [out = -120, in = -20] node {\tiny$3$} (24)

 	;

        \end{tikzpicture}
		\caption{DFA with respect to the weighted pair graph of logical control network
		\eqref{eqn1_reconstructibility}
		with $N$ equal to $4$ generated by Algorithm \ref{alg5_reconstructibility},
	where number $ij$ in each rectangle denotes state pair $\{\dt_4^i,
	\dt_4^j\}$, number $k$ beside each edge denotes
	weight $\{\dt_4^{k}\}$ of the corresponding edge.}
	\label{fig6:reconstructibility}
\end{figure}

\end{example}

\begin{remark}
	The algorithm designed in \cite{Shu2007Detectability_DES} for determining the detectability
	of DESs can also be used to determine the reconstructibility of BCNs, and 
	the algorithm designed in
	\cite{Shu2007Detectability_DES} runs in time doubly exponential in
	the number of state nodes of BCNs due to the exhaustive solution space (i.e., the set of finite input 
	sequences) search. While
	Algorithm \ref{alg4_reconstructibility} runs in time exponential in the number of state nodes of BCNs,
	because it directly depends on the structural information of the targeted BCN.
\end{remark}

\section{How to determine the reconstructibility studied in
\cite{Fornasini2013ObservabilityReconstructibilityofBCN} by using
the weighted pair graph}
\label{sec:Fornasini'sReconstructibility}

In this section we give a theorem for determining the reconstructibility
given in \cite[Definition 4]{Fornasini2013ObservabilityReconstructibilityofBCN}
(i.e., Definition \ref{def2_reconstructibility}) by using the concept of
weighted pair graphs. The theories of finite automata and formal languages are not
necessary. Hence
determining the reconstructibility in the sense of Definition \ref{def2_reconstructibility} is much easier
than determining the one in the sense of Definition \ref{def1_reconstructibility}.
One will also see that the condition in Theorem \ref{thm3_reconstructibility}
is simpler to be checked than the one given in
\cite[Theorem 4]{Fornasini2013ObservabilityReconstructibilityofBCN}, although the algorithms given in
Theorem \ref{thm3_reconstructibility} and \cite[Theorem 4]{Fornasini2013ObservabilityReconstructibilityofBCN}
have almost the same computational complexity.

\cite[Theorem 4]{Fornasini2013ObservabilityReconstructibilityofBCN} shows that
a BCN \eqref{BCN2} is reconstructible in the sense of Definition \ref{def2_reconstructibility}
iff for every pair of different periodic state-input trajectories of the same minimal period $k$
and the same input trajectory, the corresponding output trajectories are also different and periodic of 
minimal period $k$. Actually an upper bound on $k$ is $NM$, because there are totally $NM$ (state, input)-pairs.
Besides, one easily sees that Definition \ref{def2_reconstructibility}
is equivalent to the strong detectability defined in \cite{Shu2007Detectability_DES}.

\begin{definition}[\cite{Fornasini2013ObservabilityReconstructibilityofBCN}]
	A BCN \eqref{BCN2} is called reconstructible, if there is $p\in\N$ such that
	for each input sequence $U\in(\Dt_M)^{p}$, $U$ and the corresponding
	output sequence uniquely determine the current state.
	\label{def2_reconstructibility}
\end{definition}

From the proof of \cite[Theorem 4]{Fornasini2013ObservabilityReconstructibilityofBCN},
we  use graph theory to show an intuitive description of this theorem.
For a given BCN \eqref{BCN2}, construct a weighted directed graph (called {\it state transition graph})
$(V,E,W,2^{\Dt_M})$, where $V=\{(x,Hx)|x\in\Dt_N\}$,  for all $(x,Hx),(x',Hx')\in V$,
$( (x,Hx),(x',Hx'))\in E$ iff there is $u\in\Dt_M$ such that $x'=Lux$, for all $( (x,Hx),(x',\\Hx'))\in E$,
$W(( (x,Hx),(x',Hx')))=\{u\in\Dt_M|x'=Lux\}$. Then the BCN \eqref{BCN2} is
reconstructible iff in its state transition graph, every pair of different cycles with the same
length and the same input sequence have different output sequences.
Take the logical control network \eqref{eqn6_reconstructibility} with $N$
equal to $4$ for example, the state transition graph is depicted in Fig. \ref{fig5:reconstructibility}.
From Fig. \ref{fig5:reconstructibility} we find two distinct cycles $\dt_4^3\xrightarrow[]{\dt_4^2}
\dt_4^2\xrightarrow[]{\dt_4^2}\dt_4^3$
and $\dt_4^2\xrightarrow[]{\dt_4^2}\dt_4^3\xrightarrow[]{\dt_4^2}\dt_4^2$
with the same length $2$, the same input sequence
$\dt_4^2\dt_4^2$, and the same output sequence $\dt_4^1\dt_4^1$. Then this logical control network
is not reconstructible in the sense of
Definition \ref{def2_reconstructibility}.
We have shown that this logical control network is reconstructible in the sense of Definition
\ref{def1_reconstructibility} (see Example \ref{exam6_reconstructibility}),
hence these two types of reconstructibility are {\it not equivalent}.

	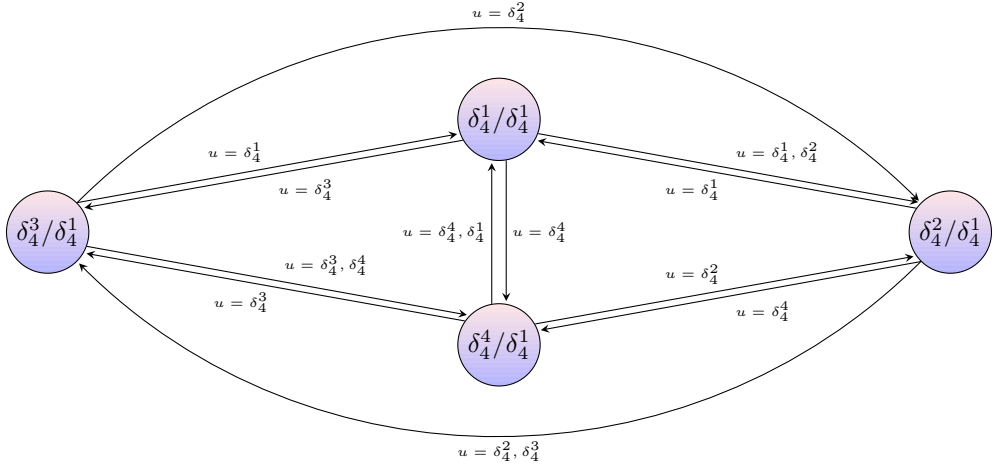
\begin{figure}
        \centering
\begin{tikzpicture}[>=stealth',shorten >=1pt,auto,node distance=1.5 cm, scale = 1.0, transform shape,
	->,>=stealth,inner sep=2pt,state/.style={shape=circle,draw,top color=red!10,bottom color=blue!30},
	skip loop/.style={to path={-- ++(0,#1) -| (\tikztotarget)}}]

	\node[state] (1) {$\dt_4^1/\dt_4^1$};
	\node (1') [below of = 1] {};
	\node (2') [right of = 1'] {};
	\node (2'') [left of = 1'] {};
	\node (3') [right of = 2'] {};
	\node (3'') [left of = 2''] {};
	\node (4') [right of = 3'] {};
	\node (4'') [left of = 3''] {};
	\node[state] (2) [right of = 4']  		  {$\dt_4^2/\dt_4^1$};
	\node[state] (3) [left of = 4'']         {$\dt_4^3/\dt_4^1$};
	\node[state] (4) [below of = 1']         {$\dt_4^4/\dt_4^1$};

	\path [->]
	(2.145) edge node {\tiny$u=\dt_4^1$} (1.-30)
	(1.-20) edge node {\tiny$u=\dt_4^1,\dt_4^2$} (2.135)
	(1.-150) edge node {\tiny$u=\dt_4^3$} (3.35)
	(3.45) edge node {\tiny$u=\dt_4^1$} (1.-160)
	(2.-135) edge node {\tiny$u=\dt_4^4$} (4.20)
	(4.30) edge node {\tiny$u=\dt_4^2$} (2.-145)
	(3.-20) edge node {\tiny$u=\dt_4^3,\dt_4^4$} (4.135)
	(4.145) edge node {\tiny$u=\dt_4^3$} (3.-30)
	(1.-80) edge node {\tiny$u=\dt_4^4$} (4.80)
	(4.100) edge node {\tiny$u=\dt_4^4,\dt_4^1$} (1.-100)
	(2) edge [out = -135, in = -45] node {\tiny$u=\dt_4^2,\dt_4^3$} (3)
	(3) edge [out = 45, in = 135] node {\tiny$u=\dt_4^2$} (2)

		 	;

        \end{tikzpicture}
		\caption{State transition graph of logical control network
		\eqref{eqn6_reconstructibility}
		with $N$ equal to $4$, where the vector in each circle denotes a state and the output it produces,
		and each weight denotes the corresponding inputs.}
	\label{fig5:reconstructibility}
\end{figure}

Note that using \cite[Theorem 4]{Fornasini2013ObservabilityReconstructibilityofBCN}
to verify the reconstructibility in the sense of Definition \ref{def2_reconstructibility},
one needs to test every pair of different cycles with the same length and the same input sequence.
Next we give a new equivalent condition for this type of reconstructibility that is simpler to
be checked than the condition in \cite[Theorem 4]{Fornasini2013ObservabilityReconstructibilityofBCN}.
Although usually the size of the weighted pair graph is larger than that of the state transition graph,
one will know whether the
BCN is reconstructible or not once the weighted pair graph has been constructed. Our result is stated as follows.

\begin{theorem}
	A BCN \eqref{BCN2} is not reconstructible in the sense of Definition \ref{def2_reconstructibility}
	iff the weighted pair graph of the BCN has a cycle.
	\label{thm3_reconstructibility}
\end{theorem}

\begin{proof}
	``if'': Assume that the weighted pair graph $\Gr=(\V,\E,\W)$ of the BCN \eqref{BCN2} has a cycle.
	$\Gr\ne\emptyset$ implies that $H$ is not of full column rank. Due to the existence of a cycle,
	there is a path $v_0\rightarrow v_1\rightarrow\dots\rightarrow
	(v_{p+1}\rightarrow\dots\rightarrow v_{p+T+1})^{\infty}$ of length $\infty$
	in the graph, where $p,T\in\N$,
	$(\cdot)^{\infty}$ means the {\it concatenation} of
	infinitely many copies of $\cdot$, and some of $v_0,\dots,v_{p+T+1}$
	may be the same. For each $i\in[0,p+T]$, choose $u_i$ from $\W( (v_i,v_{i+1}))$, and
	choose $u_{p+T+1}$ from $\W( (v_{p+T+1},v_{p+1}))$.
	Then the infinite input sequence $u_0\dots u_p(u_{p+1}\\
	\dots u_{p+T+1})^{\infty}=:U\in(\Dt_M)^{\N}$
	satisfies that for all $s\in\N$, the finite input sequence $U[0,s]$
	and the corresponding output sequence cannot uniquely determine the current state.
	Hence the BCN is not reconstructible.

	``only if'': Assume that the weighted pair graph $\Gr=(\V,\E,\W)$ of
	the BCN \eqref{BCN2} has no cycles. If $\Gr=\emptyset$, then $H$ is of full column rank, and
	the BCN is reconstructible. Next we assume that $\Gr\ne\emptyset$.

	Since $\Gr$ has no cycles, all paths are of finite length. Denote the length of the longest
	paths by $p$. Then for all input sequences $U\in(\Dt_M)^{p+1}$ and distinct states
	$x_0,\bar x_0\in\Dt_N$ satisfying $Hx_0=H\bar x_0$,
	$L_{x_0}^{p+1}(U)[p+1]\ne L_{\bar x_0}^{p+1}(U)[p+1]$ implies
	that $(HL)_{x_0}^{p+1}(U)\ne(HL)_{\bar x_0}^{p+1}(U)$.
	Hence the BCN is reconstructible.
\end{proof}

\begin{example}
	Consider the set of logical control networks \eqref{eqn6_reconstructibility}.
	In \eqref{item4_reconstructibility_worst} and \eqref{item5_reconstructibility_worst}
	of Example \ref{exam6_reconstructibility}
	we have shown that there exist cycles in the weighted pair graphs of these logical
	control networks. Then by Theorem \ref{thm3_reconstructibility}, none of these networks
	is reconstructible in the sense of Definition \ref{def2_reconstructibility}.
\end{example}

Next we analyze the computational complexity of using Theorem
\ref{thm3_reconstructibility} or \cite[Theorem 4]{Fornasini2013ObservabilityReconstructibilityofBCN}
to determine the reconstuctibility of the BCN \eqref{BCN2}
in the sense of Definition \ref{def2_reconstructibility}.
It is well known that the computational complexity of detecting the existence of a cycle in a directed graph 
$(\V,\E)$ is $O(|\V|+|\E|)$ \cite{Even2012GraphAlgorithms}.
Then the computational complexity of using Theorem \ref{thm3_reconstructibility} 
to determine this type of reconstructibility is $O(2|\V|+2|\E|)=O(2^{2n}+2^{2n+m})$, which is
proportional to the size of the weighted pair graph of the BCN. The computational complexity of using
\cite[Theorem 4]{Fornasini2013ObservabilityReconstructibilityofBCN} 
to determine this type of reconstructibility is $O(2^{2n-1+m+q})$,
because constructing the state transition graph costs $2^{n+1}+2^{n+m}$, and in the worst case that
every state forms a self-loop, checking cycles costs $(2^{2n-1}-2^{n-1})2^{m+q}$.

%
%
%
%
%
%

\section{Concluding remarks}\label{sec:conclusions}

In this paper, we defined a new weighted pair graph for a Boolean control network (BCN),
and used the graph to
design algorithms for determining a new reconstructibility (Definition \ref{def1_reconstructibility}) of BCNs.
We first designed a doubly exponential time algorithm for determining the reconstructibility, 
and then after proving more properties of the weighted pair graph, we found an exponential time algorithm.
After that, we showed that it is {\bf NP}-hard to determine the reconstructibility.

Besides, by using the weighted pair graph, we also designed an exponential time algorithm
for determining an existing reconstructibility (Definition \ref{def2_reconstructibility}) of BCNs.
It is easy to see that a BN \eqref{eqn:SATtoReconstructBN} is reconstructible in the sense of Definition
\ref{def1_reconstructibility} iff it is reconstructible in the sense of Definition \ref{def2_reconstructibility},
hence by the similar procedure as in Subsection \ref{subsec:complexityanalysis}
we have determining the reconstructibility of the BCN \eqref{BCN1} in the sense of Definition 
\ref{def2_reconstructibility} is also {\bf NP}-hard.

\section*{Acknowledgments}
The authors would like to thank Prof. W. Murray Wonham from University of Toronto,
Canada for the specific discussions on the new concept of reconstructibility and the
computational complexity of the new algorithm presented in this paper,
and also would like to express sincere gratitude to the referees and the AE
for their comments that improve the readability of the paper.

\begin{thebibliography}{99}











\bibitem{Kari2013LectureNoteonAFL}
J. Kari.
{\it A Lecture Note on Automata and Formal Languages}.
\url{http://users.utu.fi/jkari/automata/}, 2016.

\bibitem{Linz2006FormalLanguagesandAutomata}
P. Linz.
{\it An Introduction to Formal Languages and Automata}.
Sudbury, Mass. : Jones and Bartlett Publishers,
2013.


\bibitem{Ideker(2001)}
     T. Ideker, T. Galitski, L. Hood (2001).
     A new approach to decoding life: systems biology,
     {\it Annu. Rev. Genomics Hum. Genet., 2,} 343--372.

  \bibitem{Kitano2002SystemsBiology}
     H. Kitano (2002).
     Systems Biology: A brief overview,
     {\it Sceince, 259,}  1662--1664.

  \bibitem{Kauffman1969RandomBN}
     S. A. Kauffman (1969).
     Metabolic stability and epigenesis in randomly constructed genetic nets,
     {\it J. Theoretical Biology, 22,}  437--467.

  \bibitem{Akutsu2000InferringQualitativeRelations}
     T. Akutsu, S. Miyano, S. Kuhara (2000).
     Inferring qualitative relations in genetic networks and metabolic pathways,
     {\it Bioinformatics, 16,}  727--734.

  \bibitem{Albert2000DynamicsofComplexSystems}
     R. Albert, A-L Barabasi (2000).
     Dynamics of complex systems: scaling laws or the period of Boolean networks,
     {\it Phys. Rev. Lett., 84,}  5660--5663.

\bibitem{Laschov2013Minimum-timecontrolBN}
D. Laschov, M. Margaliot (2013).
Minimum-time control of Boolean networks,
{\it SIAM J. Control Optim.},  51(4), 2869--2892.

\bibitem{Zhang2015Invertibility_BCN}
K. Zhang, L. Zhang, L. Xie (2015).
Invertibility and nonsingularity of Boolean control networks,
{\it Automatica, 60,} 155--164.






  \bibitem{Akutsu(2007)}
     T. Akutsu, M. Hayashida, W. Ching, M. K. Ng (2007).
     Control of Boolean networks: Hardness results and algorithms for tree structured networks,
     {\it J. Theoretical Biology, 244,}  670--679.



  \bibitem{Cheng2010bn_dynamics}
     D. Cheng, H. Qi (2010).
     A linear representation of dynamics of Boolean networks,
     {\it IEEE Trans. Automat. Control, 55(10),}  2251--2258.


  \bibitem{Cheng2009bn_ControlObserva}
     D. Cheng, H. Qi (2009).
     Controllability and observability of Boolean control networks,
     {\it Automatica, 45(7),}  1659--1667.


  \bibitem{Zhao2010InputStateIncidenceMatrix}
     Y. Zhao, H. Qi, D. Cheng (2010).
     Input-state incidence matrix of Boolean control networks and its applications,
     {\it Systems Control Lett., 59,}  767--774.

\bibitem{Li2015ControlObservaBCN}
R. Li, M. Yang, T. Chu (2015).
Controllability and observability of Boolean networks arising from biology,
{\it Chaos: An Interdisciplinary Journal of Nonlinear Science , 25}, 023104:1--15.



\bibitem{Cheng2015NoteonObservabilityBCN}
D. Cheng, H. Qi, T. Liu, Y. Wang (2015).
A note on observability of Boolean control networks,
{\it The 10th Asian Control Conference.}




  \bibitem{Cheng2011IdentificationBCN}
     D. Cheng, Y. Zhao (2011).
     Identification of Boolean control networks,
     {\it Automatica, 47,}  702--710.



   \bibitem{Laschov2013ObservabilityofBN:GraphApproach}
	   D. Laschov, M. Margaliot, G. Even (2013).
	   Observability of Boolean networks: A graph-theoretic approach,
	   {\it Automatica, 49,} 2351--2362.



   \bibitem{Fornasini2013ObservabilityReconstructibilityofBCN}
 E. Fornasini, M. Valcher    (2013).
 Observability, reconstructibility and state observers of Boolean control networks,
{\it IEEE Trans. Automat. Control, 58(6),}  1390--1401.

\bibitem{Fornasini2015FaultDetectionBCN}
E. Fornasini, M. Valcher (2015).
Fault detection analysis of Boolean control networks,
{\it IEEE Trans. Automat. Control, 60(10)}, 2734--2739.


\bibitem{Li2013ObservabilityConditionsofBCN}
R. Li, M. Yang, T. Chu (2013).
Observability conditions of Boolean control networks,
{\it Int. J. Robust. Nonlinear Control}, 24, 2711--1723.




  \bibitem{Cheng(2011book)}
     D. Cheng, H. Qi, Z. Li (2011).
     {\it Analysis and Control of Boolean Networks: A Semi-tensor Product
     Approach,} London: Springer.







\bibitem{Zhang2014ObservabilityofBCN}
          K. Zhang, L. Zhang (2014).
         Observability of Boolean control networks: A unified approach based on the theories of finite automata and formal languages,
         the 33rd Chinese Control Conference, Nanjing, China, July 28--30, 6854--6861, its extension has been
		 accepted by {\it IEEE Trans. Automat. Control}.

	 \bibitem{Zhang2016ObservabilitySBCN_NAHS}
		 K. Zhang, L. Zhang, L. Xie (2016).
Finite automata approach to observability of switched Boolean control networks,
{\it Nonlinear Anal. Hybrid Syst.}, 19, 186--197.




\bibitem{Shu2007Detectability_DES}
S. Shu, F. Lin, H. Ying (2007).
Detectability of Discrete Event Systems,
{\it IEEE Trans. Automat. Control, 52(12),}  2356--2359.





	 \bibitem{Ramadge1987SupervisoryControlofDES}
P. Ramadge, W. Wonham (1987).
Supervisory control of a class of discrete event processess,
  {\it SIAM J. Control Optim.,} 25:206--230.


\bibitem{Cassandras2008DES_monograph}
C. G. Cassandras, S. Lafortune (2008).
{\it Introduction to Discrete Event Systems}, Springer US.


  \bibitem{Lin1988Observability_DES}
F. Lin, W. Wonham (1988).
On observability of discrete-event systems,
{\it Information Sciences}, 44, 173--198.

\bibitem{Wonham1976InternalModelPriniple}
W. Wonham (1976).
Towards an abstract internal model principle,
{\it IEEE Trans. Syst., Man, Cybern., SMC-6(11)}, 735--740.

\bibitem{Lin2014LinControlofNetworkedDES}
F. Lin (2014).
Control of networked discrete event systems:
dealing with communication delays and losses,
{\it SIAM J. Control Optim.}, 52(2), 1276--1298.


\bibitem{Wonham2014SupervisoryControlDES}
W. Wonham (2014)
Supervisory Control of Discrete-Event Systems,
lecture note,
\url{http://www.control.utoronto.ca/DES/}.


\bibitem{Su2011StringExecutionTime}
R. Su, G. Woeginger (2011).
String execution time for finite languages: Max is easy, min is hard,
{\it Automatica 47}, 2326--2329.








\bibitem{Even2012GraphAlgorithms}
S. Even, G. Even (2012).
{\it  Graph Algorithms}.
2nd edition, Cambridge University Press.


\bibitem{Sridharan2012BNFaultDiagnosisOSR}
S. Sridharan, R. Layek, A. Datta, J. Venkatraj (2012).
Boolean modeling and fault diagnosis in oxidative stress response,
{\it BMC Genomics}, 13(Suppl 6): S4, 1--16.


\end{thebibliography}
\end{document}